\documentclass[a4paper]{amsart}

\usepackage[T1]{fontenc}
\usepackage[utf8]{inputenc}
\usepackage[english]{babel}
\usepackage[includeheadfoot, margin=1in]{geometry}
\usepackage{ifthen}

\usepackage[bookmarksnumbered]{hyperref}

\usepackage{amsmath}
\usepackage{amssymb}
\usepackage{mathrsfs}
\usepackage{eucal}			
\usepackage{tensor}

\usepackage{amsthm}
\usepackage{enumitem}  
\usepackage[normalem]{ulem} 

\usepackage{tikz}  \usetikzlibrary{arrows,positioning}

\usepackage[numbers]{natbib}
\usepackage{doi}

\newcounter{generalnumbering}   \numberwithin{generalnumbering}{section}

\theoremstyle{plain}    \newtheorem{theorem}[generalnumbering]{Theorem}
\theoremstyle{plain}    \newtheorem{corollary}[generalnumbering]{Corollary}
\theoremstyle{definition}   \newtheorem{definition}[generalnumbering]{Definition}
\theoremstyle{definition}   \newtheorem{example}[generalnumbering]{Example}
\theoremstyle{plain}    \newtheorem{proposition}[generalnumbering]{Proposition}
\theoremstyle{plain}    \newtheorem{lemma}[generalnumbering]{Lemma}

\newcommand{\namefordifferentenvironment}{}

\theoremstyle{plain}    \newtheorem{plainstyle}[generalnumbering]{\namefordifferentenvironment}
\theoremstyle{plain}    \newtheorem*{plainstyle*}{\namefordifferentenvironment}
\theoremstyle{definition}    \newtheorem{definitionstyle}[generalnumbering]{\namefordifferentenvironment}
\theoremstyle{definition}    \newtheorem*{definitionstyle*}{\namefordifferentenvironment}

\newenvironment{penv*}[1]{\renewcommand{\namefordifferentenvironment}{#1}\begin{plainstyle*}}{\end{plainstyle*}}

\newenvironment{denv*}[1]{\renewcommand{\namefordifferentenvironment}{#1}\begin{definitionstyle*}}{\end{definitionstyle*}}

\newenvironment{remark}{\begin{denv*}{Remark}}{\end{denv*}}

\newcommand{\ntag}{\tag{\thegeneralnumbering}\stepcounter{generalnumbering}}

\DeclareMathOperator{\so}{\mathfrak{s}}
\DeclareMathOperator{\ra}{\mathfrak{r}}
\DeclareMathOperator{\supp}{supp}
\DeclareMathOperator{\id}{id}
\DeclareMathOperator{\dom}{dom}
\DeclareMathOperator{\ran}{ran}

\newcommand*{\defeq}{\mathrel{\vcenter{\baselineskip0.5ex \lineskiplimit0pt\hbox{\scriptsize.}\hbox{\scriptsize.}}}=}

\title{Sectional algebras of semigroupoid  bundles}
\author{Luiz Gustavo Cordeiro
}
\thanks{The author was supported by the ANR project GAMME (ANR-14-CE25-0004)}

\address{UMPA, UMR 5669 CNRS -- École Normale Supérieure de Lyon\\
46 allée d'Italie, 69364 Lyon Cedex 07, France}
\email{luizgc6@gmail.com, luis-gustavo.cordeiro@ens-lyon.fr}				

\subjclass[2010]
{Primary 16S60; 
Secondary 18B40, 
13A02
}

\keywords{semigroupoid; sectional algebra; smash product; tensor product; semidirect product; crossed product; $\land$-preaction; quotient}		

\begin{document}

\begin{abstract}
    In this article we use semigroupoids to describe a notion of algebraic bundles, mostly motivated by Fell ($C^*$-algebraic) bundles, and the sectional algebras associated to them. As the main motivational example, Steinberg algebras may be regarded as the sectional algebras of trivial (direct product) bundles. Several theorems which relate geometric and algebraic constructions -- via the construction of a sectional algebra -- are widely generalized: Direct products bundles by semigroupoids correspond to tensor products of algebras; Semidirect products of bundles correspond to ``naïve'' crossed products of algebras; Skew products of graded bundles correspond to smash products of graded algebras; Quotient bundles correspond to quotient algebras. Moreover, most of the results hold in the non-Hausdorff setting. In the course of this work, we generalize the definition of smash products to groupoid graded algebras.

    As an application, we prove that whenever $\theta$ is a $\land$-preaction of a discrete inverse semigroupoid $S$ on an ample (possibly non-Hausdorff) groupoid $\mathcal{G}$, the Steinberg algebra of the associated groupoid of germs is naturally isomorphic to a crossed product of the Steinberg algebra of $\mathcal{G}$ by $S$. This is a far-reaching generalization of analogous results which had been proven in particular cases.
\end{abstract}

\maketitle

\section{Introduction}
\subsection{Historical remarks}

Bundles (or fields) of algebras have been thoroughly studied in the last century, and are an instance of the general technique of decomposing a mathematical object into more manageable components. For example, any finite-dimensional $C^*$-algebra $A$ may be decomposed as a direct sum $A=\oplus_{x\in X} A_x$, where $X$ is a finite set called the \emph{base space} and each $A_x$ is a full matrix algebra. In other words, $A$ is isomorphic to an algebra of block diagonal matrices.

However, when considering infinite dimensional algebras, such a decomposition is generally not possible in any meaningful way. There are two possibilities to deal with this problem.

In one direction, we may permit that the base space $X$, over which the initial algebra $A$ is decomposed, is infinite, or more specifically a topological space. This leads to the notion of ``continuous bundles of $C^*$-algebras''. These  have their origin in Godement \cite{MR0038571} and Kaplansky \cite{MR0042066} as a generalization of direct sums to a continuous setting, after von Neumann's introduction of measurable fields of Hilbert spaces in \cite{MR0029101}.

In the other direction, we may allow that the base space $X$ itself has some dynamical strucuture (on top of possibly being infinite) -- e.g.\ it is a group -- and the decomposition of $A$ along $X$ respects that structure, i.e., that $A$ is a \emph{graded} algebra.

In \cite{MR0259619}, Fell introduces ``Banach $*$-algebraic bundles'' over topological groups, which are continuous versions of group graded $*$-algebras. The $C^*$-analogues were introduced in \cite{MR936629}, and are now more commonly referred to as ``Fell bundles''.

Finally, both of the approaches above were combined by Kumjian in \cite{MR1443836}, who defined Fell bundles over groupoids.

The goal of this article is to utilize the language of \emph{semigroupoids} to lay out a general framework for the study of bundles of algebras. Whenever possible, we will consider general topological algebras over topological rings satisfying some minimal regularity conditions. Although the finer details become more intricate in this very general setting, it has the obvious advantage of being more widely applicable.

Let us outline the structure of this article. The remainder of the Introduction is devoted to recall some basic facts and terminology, which will be used throughout the paper, about algebras over non-commutative rings, topological spaces, and topological semigroupoids. In Section 2 we define algebraic bundles over semigroupoids and their associated sectional algebras. Section 3 describes several classes of algebras which may be regarded as somewhat trivial cases of sectional bundle algebras, so as to make it precise how to apply our results in those specific settings. The short Section 4 deals with a very natural problem: Are sectional algebras topological algebras in a natural manner? We provide a positive answer in a specific, but nonetheless sufficiently wide, setting. We finish this article with the fifth section, with several isomorphism theorems relating geometric constructions with bundles and algebraic constructions of the associated sectional algebras. The results of this section affirm the strength of this theory in their applications, since we are able to generalize and connect previously-known results of the are by interpreting different theories in this more general setting in very straightforward manners.

\subsection{Algebraic generalities}

All rings and algebras are assumed to be associative, and any module or bimodule $M$ over a unital ring $R$ will be assumed to be unital, i.e., $1m=m$ and/or $m1=m$ for all $m\in M$.

If $A$ is an abelian group, written additivelly, and $\mathscr{F}$ is a collection of subsets of $A$, then $\sum\mathscr{F}$ denotes the abelian subgroup of $A$ generated by $\mathscr{F}$, i.e., the set of all finite sums of elements of sets in $\mathscr{F}$.

Some fine details of the theory of algebras over non-commutative rings differs significantly from that of commutative rings, so let us spell out all relevant definitions.

\begin{denv*}{Tensor products}
Let $R$ be a ring, $M$ a right $R$-module and $N$ a left $R$-module. If $A$ is an abelian (additive) group, a map $T\colon M\times N\to A$ is said to be \emph{balanced} if
\begin{itemize}
    \item For every $m\in M$ and $n\in N$, the sections $T(\cdot,n)\colon M\to A$ and $T(m,\cdot)\colon N\to A$ are additive;
    \item For every $m\in M$, $n\in N$ and $r\in R$, we have $T(mr,n)=T(m,rn)$
\end{itemize}

The \emph{tensor product} $M\tensor[_R]{\otimes}{_R}N$ is constructed in the usual manner, as the free abelian group generated by symbols $m\otimes n$, where $(m,n)\in M\times N$, modulo the conditions stating that the map $T\colon M\times N\to M\tensor[_R]{\otimes}{_R}N$ is balanced. If no confusion arises we write simply $M\otimes N$.

If $S$ is another ring and $M$ has an $(S,R)$-bimodule structure, then the tensor product $M\tensor[_R]{\otimes}{_R}N$ has a left $S$-module structure determined by $s(m\otimes n)=(sm)\otimes n$ for all $s\in S$ and $(m,n)\in M\times N$.

Similarly, any right module structure on $N$ compatible with the left $R$-module structure induces a right module structure on $M\tensor[_R]{\otimes}{_R}N$.
\end{denv*}

\begin{denv*}{Algebras over non-commutative rings}
An \emph{algebra} $A$ over a ring $R$ (or simply an \emph{$R$-algebra}) consists of a ring $A$ enriched with an $R$-bimodule structure -- the ring addition and the $R$-bimodule addition being the same -- such that for all $a,b\in A$ and $r\in R$,
\[(ra)b=r(ab),\qquad (ar)b=a(rb),\qquad\text{and}\qquad (ab)r=a(br).\]
The middle equation above means that the product of $A$ is balanced, and thus is determined as an $R$-bimodule homomorphism $A\tensor[_R]{\otimes}{_R} A\to A$, $a\otimes b\mapsto ab$.
\end{denv*}

\begin{denv*}{Tensor products of algebras}
In general, the tensor product of two $R$-algebras, where $R$ is a ring, is just an $R$-bimodule, and not an algebra. This is in constrast with the more familiar setting of commutative rings, which we briefly recall below.

If $R$ is a commutative ring, then any left (or right) $R$-module $M$ may be regarded as an $R$-bimodule where the left and right actions of $R$ are the same: $rm=mr$ for all $r\in R$ and $m\in M$. The $R$-bimodules obtained in this manner are called \emph{symmetric}. This defines an injective and full functor from the category of left $R$-modules to the category of $R$-bimodules. However this functor is not essentially surjective, since there are non-symmetric bimodules over commutative rings (e.g.\ $R=D_2(\mathbb{R})$, the ring of $2\times 2$ real diagonal matrices, and $M=M_2(\mathbb{R})$, the $2\times 2$ real matrix algebra, regarded as an $R$-bimodule via matrix multiplication).

If $A$ and $B$ are symmetric $R$-algebras (i.e., the $R$-bimodule structures are symmetric), then the tensor product $A\tensor[_R]{\otimes}{_R}B$ has a canonical symmetric $R$-algebra structure, determined by $(a_1\otimes b_1)(a_2\otimes b_2)=(a_1a_2)\otimes (b_1b_2)$.
\end{denv*}

\subsection{Topological conventions}

A subset $A\subseteq X$ is a \emph{neighbourhood} of a point $x\in X$ if $x$ belongs to its interior $\operatorname{int}(A)$. A \emph{neighbourhood basis} of $x$ is a set $\mathscr{B}$ of neighbourhoods of $x$ such that for any neighbourhood $A$ of $x$, there exists $B\in\mathscr{B}$ such that $B\subseteq A$.

Whenever we state that a topological space $X$ satisfies some property ``locally'', we shall mean that every point of $X$ admits a neighbourhood basis consisting of subsets satisfying such property under the subspace topology. So for example, a space $X$ is locally compact if every point of $X$ admits a neighbourhood basis of compact sets, which will be in general not open in $X$. If necessary for precision, we may say that a topological property holds ``globally'' in constrast with ``locally'' (e.g.\ \emph{globally Hausdorff}).

Even though this ``local requirement'' may  sometimes be weakened to just assuming that each point has at least one neighbourhood (instead of a basis) satisfying such a property, this is not sufficiently strong for our needs. For example, a compact and locally Hausdorff space $X$ may be non-locally compact in our sense (although it is well-known that any compact and globally Hausdorff space is so).

If $X$ and $Y$ are topological spaces, the set of continuous functions from $X$ to $Y$ is denoted by $C(X,Y)$. If $A$ is a topological ring or a topological algebra, the support $\supp(f)$ of a (possibly discontinuous) function $f\colon X\to A$ is the closure of $\left\{x\in X:f(x)\neq 0\right\}$.

\subsection{Topological and étale semigroupoids}

Semigroupoids provide a modern approach to unify the theories of categories and semigroups, and in particular the study of \emph{inverse semigroupoids} allows one to join the theories of groupoids and inverse semigroups. Two working definitions, by Tilson \cite{MR915990} and Exel \cite{MR2419901}, have appeared in the literature. For our purposes, we will consider semigroupoids in the sense of Tilson, which are, in simple terms, ``categories without identities''.

\begin{definition}\label{def:semigroupoid}
    A \emph{semigroupoid} consists of a tuple $(\Lambda,\Lambda^{(0)},\Lambda^{(2)},\so,\ra,\mu)$, where
    \begin{enumerate}[label=(\roman*)]
        \item\label{def:semigroupoid.graph} $\Lambda$ is a directed graph (or quiver) over $\Lambda^{(0)}$, with source and range maps $\so,\ra\colon\Lambda\to\Lambda^{(0)}$, respectively; (we allow loops and multiple arrows between vertices)
        \item\label{def:semigroupoid.lambda2} $\Lambda^{(2)}=\left\{(a,b)\in\Lambda\times\Lambda:\so(a)=\ra(b)\right\}$ is the set of \emph{composable pairs}. Note that this is also the graph of $2$-paths on $\Lambda$;
        \item\label{def:semigroupoid.product} $\mu\colon\Lambda^{(2)}\to\Lambda$ is the \emph{multiplication} or \emph{product} map, and denoted by concatenation -- $\mu(a,b)=ab$;
        \item\label{def:semigroupoid.respect.sources.and.ranges} $\mu$ is a graph morphism from $\Lambda^{(2)}$ to $\Lambda$, i.e., if $\so(a)=\ra(b)$ then $\so(ab)=\so(b)$ and $\ra(ab)=\ra(a)$;
        \item\label{def:semigroupoid.associativity} $\mu$ is associative, i.e., $(ab)c=a(bc)$ whenever $\so(a)=\ra(b)$ and $\so(b)=\ra(c)$.
    \end{enumerate}
\end{definition}

The product of subsets $A,B$ of a semigroupoid $\Lambda$ is
\[AB\defeq\left\{ab:(a,b)\in(A\times B)\cap\Lambda^{(2)}\right\}.\]

A map $\phi\colon\Lambda_1\to\Lambda_2$ between semigroupoids is a \emph{homomorphism} if $(a,b)\in\Lambda_1^{(2)}$ implies $(\phi(a),\phi(b))\in\Lambda_2^{(2)}$ and $\phi(ab)=\phi(a)\phi(b)$. An \emph{isomorphism} is a bijective homomorphism whose inverse is also a homomorphism.

More generally, in the same manner that one may define a category internal to any category with pullbacks, we may also define semigroupoids internal to any category with pullbacks. In particular, a \emph{topological semigroupoid} is simply a semigroupoid $\Lambda$ where both $\Lambda$ and $\Lambda^{(0)}$ are endowed with certain topologies making all structural maps (source, range, and multiplication) continuous. (In this case $\Lambda^{(2)}$ has the product topology, coming from $\Lambda\times\Lambda$.)

\begin{definition}
    An \emph{étale} semigroupoid is a topological semigroupoid $\mathcal{E}$ such that the source and range maps $\so,\ra\colon\mathcal{E}\to\mathcal{E}^{(0)}$ are local homeomorphisms and the vertex set $\mathcal{E}^{(0)}$ is locally compact and globally Hausdorff.
\end{definition}

Note that an étale semigroupoid $\mathcal{E}$ is locally compact and locally Hausdorff. The product map of $\mathcal{E}$ is also a local homeomorphism, so $\mathcal{E}$ is actually a semigroupoid internal to the category of topological spaces and étale maps (local homeomorphisms). This fact may be proven just as in \cite[Proposition 3.5]{arxiv1902.09375}. Note that the Hausdorff property of $\mathcal{E}^{(0)}$ is necessary just to ensure that the product of compact sets is compact, which can also be proven just as in \cite[Lemma 5.1(b)]{arxiv1902.09375}. In short, we have:

\begin{proposition}\label{prop:product.map.is.local.homeo}
    If $\mathcal{E}$ is an étale semigroupoid, then the product of compact subsets of $\mathcal{E}$ is compact, and the product is a local homeomorphism from $\mathcal{E}^{(2)}$ to $\mathcal{E}$.
\end{proposition}

If $\mathcal{E}$ is an étale semigroupoid, then $\mathcal{E}$ admits a basis of open subsets $U$ such that the source and range maps restrict to homeomorphisms of $U$ onto open subsets of $\mathcal{E}^{(0)}$. These sets are called the \emph{open bisections} of $\mathcal{E}$, and will be used heavily throughout this article.

\begin{definition}
   A \emph{bisection} of a semigroupoid $\mathcal{E}$ is a subset $U\subseteq\mathcal{E}$ such that the source and range maps are injective on $U$. If $\mathcal{E}$ is étale, we denote by $\mathbf{B}(\mathcal{E})$ the set of all open bisections of $\mathcal{E}$. In this case, $\mathbf{B}(\mathcal{E})$ is a topological basis for $\mathcal{E}$, and it is closed under products of sets (and hence is a semigroup).
\end{definition}

We will also be interested more specifically in \emph{inverse semigroupoids}, which allow us to consider dynamical systems by means of (global/partial/$\land$-pre-) actions. We refer to \cite{arxiv1902.09375} for the finer details, but nevertheless let us write all relevant definitions and properties below.

\begin{definition}
    An \emph{inverse semigroupoid} is a semigroupoid $\mathcal{S}$ such that for every $s\in\mathcal{S}$, there exists a unique $t\in\mathcal{S}$ such that $\so(s)=\ra(t)$, $\ra(s)=\so(t)$, $sts=s$ and $tst=t$. This unique element is denoted $t=s^*$ and is called the \emph{inverse} of $s$.
\end{definition}

The two best examples of (étale) inverse semigroupoids to keep in mind are (discrete) inverse semigroups and (étale) groupoids.

Just as in the case for inverse semigroups, we denote by $E(\mathcal{S})\defeq\left\{e\in\mathcal{S}:\so(e)=\ra(e)\text{ and }ee=e\right\}$ the set of \emph{idempotents} of an inverse semigroupoid $\mathcal{S}$. Then $E(\mathcal{S})$ is a commutative subsemigroupoid of $\mathcal{S}$, i.e., if $e,f\in E(\mathcal{S})$, then $ef$ is defined if and only if $fe$ is defined, in which case $ef=fe$.

We have a canonical order on any inverse semigroupoid $\mathcal{S}$, where $s\leq t$ is determined by any of the following equivalent statements: (i) $s=ts^*s$; (ii) $s=te$ for some $e\in E(\mathcal{S})$; (iii) $s=ss^*t$; or (iv) $s=ft$ for some $f\in E(\mathcal{S})$. (Note that we implicitly assume that $\so(s)=\so(t)$ and $\ra(s)=\ra(t)$.)

The inverses and the order in inverse semigroupoids obey the usual rules: As long as the statements make sense, we have (a) $(st)^*=t^*s^*$; (b) $(s^*)^*=s$; (c) $s\leq t\iff s^*\leq t^*$; and (d) $s_1\leq t_1$ and $s_2\leq t_2$ implies $s_1s_2\leq t_1t_2$.

For topological (and étale) inverse semigroupoids, we also assume that the inversion map $a\mapsto a^*$ is continuous. In this case, the semigroup $\mathbf{B}(\mathcal{S})$ of open bisections is an inverse semigroup.

\subsection{\texorpdfstring{$\land$}{∧}-preactions, partial actions, and global actions}

The notions of partial and global actions (and the more general but less studied $\land$-preactions) of inverse semigroups and groupoids can be immediatelly generalized to the context of inverse semigroupoids. As we want to have a general approach that encompasses both the topological as the algebraic settings, it is useful to consider actions of inverse semigroupoids on semigroupoids.

If $f$ is a function, we denote its domain by $\dom(f)$ and its range by $\ran(f)$. If $g$ is another function, then the composition $gf$ is defined on ``the largest domain on which the formula $(gf)(x)=g(f(x))$ makes sense'', that is,
\[\dom(gf)\defeq f^{-1}(\ran(f)\cap\dom(g)),\qquad\ran(gf)\defeq g(\dom(g)\cap\ran(f)),\]
and $(gf)(x)=g(f(x))$ for all $x\in\dom(gf)$.

\begin{definition}
    An \emph{ideal} of a semigroupoid $\Lambda$ is a subset $I\subseteq\Lambda$ such that $I\Lambda\cup\Lambda I\subseteq I$.
\end{definition}

\begin{definition}\label{def:action}
    A \emph{$\land$-preaction} $\theta$ of an inverse semigroupoid $\mathcal{S}$ on a semigroupoid $\Lambda$ consists of a collection of maps $\left\{\theta_s\right\}_{s\in\mathcal{S}}$ satisfying:
    \begin{enumerate}[label=(\roman*)]
        \item\label{def:action.big.ideal} For all $v\in\mathcal{S}^{(0)}$, $\bigcup_{s\in\so^{-1}(v)}\dom(\theta_s)$ is an ideal of $\Lambda$, which we temporarily denote $I(\theta,v)$;
        \item\label{def:action.domain.is.ideal} For all $s\in\mathcal{S}$, $\dom(\theta_s)$ is an ideal of $I(\theta,\so(s))$, $\ran(\theta_s)$ is an ideal of $I(\theta,\ra(s))$, and $\theta_s$ is a semigroupoid isomorphism from $\dom(\theta_s)$ to $\ran(\theta_s)$;
        \item For all $s\in\mathcal{S}$, $\theta_{s^*}=\theta_s^{-1}$ (in particular $\ran(\theta_s)=\dom(\theta_{s^*}))$;
        \item\label{def:action.inclusion.of.domains} If $(s,t)\in\mathcal{S}^{(2)}$, then $\theta_{st}$ is an extension of $\theta_s\theta_t$, i.e., $\theta_t^{-1}(\ran(\theta_t)\cap\dom(\theta_s))\subseteq\dom(\theta_{st})$ and $\theta_s(\theta_t(x))=\theta_{st}(x)$ for all $x\in\theta_t^{-1}(\ran(\theta_t)\cap\dom(\theta_s))$.
    \end{enumerate}
    
    A $\land$-preaction is called a \emph{partial action} if it satisfies, in addition:
    \begin{enumerate}[label=(\roman*)]\setcounter{enumi}{4}
        \item If $s\leq t$ in $\mathcal{S}$ then $\dom(\theta_s)\subseteq\dom(\theta_t)$.
    \end{enumerate}
    Finally, $\theta$ is a \emph{global action} if $\theta_{st}=\theta_s\circ\theta_t$ for all $(s,t)\in\mathcal{S}^{(2)}$, or in other words if the set inclusion in item \ref{def:action.inclusion.of.domains} is actually an equality. Every global action is a partial action.
    
    A $\land$-preaction $\theta$ of an inverse semigroupoid $\mathcal{S}$ on a semigroupoid $\Lambda$ will be denoted by ``$\theta\colon\mathcal{S}\curvearrowright\Lambda$''.
\end{definition}

In operational terms:
\begin{itemize}
    \item If $\theta$ is a $\land$-preaction and $(s,t)\in\mathcal{S}^{(2)}$, then 
    \[\theta_s(\theta_t(x))=\theta_{st}(x)\ntag\label{eq:landpreactionproperty}\]
    whenever the \uline{left-hand side} is defined, as it implies that the right-hand side is also defined;
    \item If $\theta$ is a $\land$-preaction and $s\leq t$ in $\mathcal{S}$, then
    \[\theta_s(x)=\theta_t(x)\ntag\label{eq:landpreactionchangebygreater}.\]
    whenever \uline{both sides are simultaneously} defined.
    \item If $\theta$ is a partial action, then Equation \eqref{eq:landpreactionchangebygreater} holds whenever its \uline{left-hand side} is defined, as it implies that the right-hand side is also defined;
    \item If $\theta$ is a global action, then Equation \eqref{eq:landpreactionproperty} holds whenever \uline{any} side is defined, as it implies that the other one also defined.
\end{itemize}

\begin{remark}
    Originally (\cite[Definition 2.41]{arxiv1902.09375}), in the definition of $\land$-preactions we also require a function $\pi\colon\Lambda\to\mathcal{S}^{(0)}$ which is a semigroupoid morphism, in the sense that if $ab$ is defined in $\Lambda$ then $\pi(a)=\pi(b)$. Instead of conditions \ref{def:action.big.ideal} and \ref{def:action.domain.is.ideal} of Definition \ref{def:action}, we require that
    \begin{enumerate}[label=(\roman*)']
        \item $\pi^{-1}(v)$ is an ideal of $\Lambda$ for every $v\in\mathcal{S}^{(0)}$; and
        \item $\dom(\theta_s)$ is an ideal of $\pi^{-1}(\so(s))$ for every $s\in\mathcal{S}$.
    \end{enumerate}
    These two approaches are in fact equivalent in the following sense: Given a $\land$-preaction $\theta$ as in Definition \ref{def:action}, we construct the new semigroupoid \[\Gamma\defeq\bigcup_{v\in\mathcal{S}^{(0)}}\left\{v\right\}\times I(\theta,v)\]
    with canonical product when we see $\mathcal{S}^{(0)}$ as a unit groupoid: $(v,a)(u,b)=(v,ab)$ whenever $v=u$ and $ab$ is defined in $\Lambda$.
    
    Now we define a new action $\widetilde{\theta}$ of $\mathcal{S}$ on $\Gamma$ by setting, for every $s\in\mathcal{S}$, $\dom(\widetilde{\theta}_s)=\left\{\so(s)\right\}\times\dom(\theta_s)$, and $\widetilde{\theta}_s(\so(s),a)=(\ra(s),\theta_s(a))$. 
    
    Let $\pi_1\colon\Gamma\to\mathcal{S}^{(0)}$, $\pi_1(x,a)=x$. Then $(\pi_1,\widetilde{\theta})$ is an $\land$-preaction in the sense of \cite{arxiv1902.09375}. The second coordinate map $\pi_2\colon\Gamma\to\Lambda$, $\pi_2(v,a)=a$, intertwines the $\land$-preactions $\widetilde{\theta}$ and $\theta$, i.e., for every $s\in\mathcal{S}$, $\pi_2$ restricts to a bijection of $\dom(\widetilde{\theta}_s)$ onto $\dom(\theta_s)$, and $\pi_2\circ\widetilde{\theta}_s=\theta_s\circ\pi_2$ on $\dom(\widetilde{\theta}_s)$. All of this remains true for topological semigroupoids, in which case all relevant functions are continuous.
    
    The main point is that the semidirect product $\mathcal{S}\ltimes\Lambda$, as defined below in Definition \ref{def:semidirect.product}, is isomorphic to the semidirect product $\mathcal{S}\ltimes\Lambda$ of \cite[Definition 2.51]{arxiv1902.09375}, and thus all the theory of \cite{arxiv1902.09375} may be transferred immediately to the setting we consider.
\end{remark}

\begin{definition}\label{def:semidirect.product}
    Given a $\land$-preaction $\theta\colon\mathcal{S}\curvearrowright\Lambda$, the \emph{semidirect product} $\mathcal{S}\ltimes_\theta\Lambda$ is the set
    \[\mathcal{S}\ltimes_\theta\Lambda\defeq\bigcup_{s\in\mathcal{S}}\left\{s\right\}\times\dom(\theta_s)=\left\{(s,a)\in\mathcal{S}\times\Lambda:a\in\dom(\theta_s)\right\}\]
    with graph structure over $\mathcal{S}^{(0)}\times\Lambda^{(0)}$ given by
    \[\so(s,a)=\left(\so(s),\so(a)\right)\qquad\text{and}\qquad\ra(s,a)=\left(\ra(s),\ra\left(\theta_s(a)\right)\right)\]
    and product
    \[(s,a)(t,b)=(st,\theta_{t^*}(a\theta_t(b)))\label{eq:semidirectproduct}\]
    whenever $\so(s)=\ra(t)$ and $\so(a)=\ra(\theta_t(b))$.
\end{definition}

We write simply $\mathcal{S}\ltimes\Lambda$ for the semidirect product if no confusion arises from dropping $\theta$ from the notation.

The product of $\mathcal{S}\ltimes\Lambda$ is not associative in general, but it is associative in all cases of interest. For example, if $\dom(\theta_s)$ is an inverse semigroupoid for all $s\in\mathcal{S}$, then the product defined in Definition \ref{def:semidirect.product} is associative, and $\mathcal{S}\ltimes\Lambda$ is a semigroupoid. See \cite[Section 2.5]{arxiv1902.09375} for details.

\begin{definition}
    A $\land$-preaction $\theta\colon\mathcal{S}\curvearrowright\Lambda$ is \emph{associative} if for all $s,t,u\in\mathcal{S}$ with $stu$ defined, and all $(a,b,c)\in\dom(\theta_s)\times\dom(\theta_t)\times\ran(\theta_u)$, we have $\theta_{t^*}(a\theta_t(b))c=\theta_{t^*}(a\theta_t(bc))$. This is equivalent to the semidirect product being associative with respect to the product in Definition \ref{eq:semidirectproduct} (see the proof of \cite[Theorem 2.56]{arxiv1902.09375}).
\end{definition}

In general we will only consider associative $\land$-preactions.

\begin{definition}
    Given a $\land$-preaction $\theta\colon\mathcal{S}\curvearrowright\Lambda$, we abuse notation and also use $\theta$ to denote the \emph{action map}
    \[\theta\colon\mathcal{S}\ltimes\Lambda\to\Lambda,\quad \theta(s,a)=\theta_s(a).\]
    If $\mathcal{S}$ is a topological inverse semigroupoid and $\Lambda$ is a topological semigroupoid, we say that $\theta$ is \emph{continuous} or \emph{open} if the action map is continuous or open, respectively.
\end{definition}

Of course, if $\theta\colon\mathcal{S}\curvearrowright\Lambda$ is a continuous associative $\land$-preaction, then $\mathcal{S}\ltimes\Lambda$ is a topological semigroupoid. If $\mathcal{S}$ and $\Lambda$ are étale and $\theta$ is continuous and open, then $\mathcal{S}\ltimes\Lambda$ is étale as well.

\section{Sectional algebras}

\subsection{Algebraic bundles}

Throughout this section, we let $R$ be a fixed unital topological ring.

We will define $R$-bundles in terms of semigroupoid homomorphisms. However an additional property will be required of the homomorphisms under consideration.

Let $\pi\colon\Lambda\to\Gamma$ be a homomorphism of semigroupoids. Then we have $\Lambda^{(2)}\subseteq(\pi\times\pi)^{-1}(\Gamma^{(2)})$, however the reverse inclusion is not true in general. In general, the image $\pi(\Lambda)$ might not be a subsemigroupoid of $\Gamma$ and thus is not a semigroupoid in any natural manner. This is a realization of the fact that the kernel of $\pi$, $\ker\pi\defeq\left\{(x,y)\in\Lambda\times\Lambda:\pi(x)=\pi(y)\right\}$, is not a congruence for $\Lambda$ in any suitable sense (i.e., in a manner that the quotient $\Lambda/\ker\pi$ has a natural semigroupoid structure).

On the other hand, if $(\pi\times\pi)^{-1}(\Gamma^{(2)})\subseteq\Lambda^{(2)}$, then $\pi(\Lambda)$ is in fact a subsemigroupoid of $\Gamma$, and the quotient $\Lambda/\!\ker\pi$ has a canonical semigroupoid structure, making it isomorphic to $\pi(\Lambda)$.

\begin{definition}
    A semigroupoid homomorphism $\pi\colon\Lambda\to\Gamma$ is \emph{rigid} if $(\pi\times\pi)^{-1}(\Gamma^{(2)})=\Lambda^{(2)}$.
\end{definition}

When $\Lambda$ has no sources nor sinks as a graph, rigidity of a homomorphism $\pi\colon\Lambda\to\Gamma$ has alternative descriptions. In this case, $\pi$ induces a unique vertex map $\pi^{(0)}\colon\Lambda^{(0)}\to\Gamma^{(0)}$ in such a way that $(\pi^{(0)},\pi)$ is a graph morphism, i.e., $\so_{\Gamma}\circ\pi=\pi^{(0)}\circ\so_{\Lambda}$, and similarly for the range maps (see \cite[Proposition 2.19]{arxiv1902.09375}).

On the other hand, following \cite[4.1.2]{MR3597709}, a congruence $\rho$ on $\Lambda$ is called \emph{rigid} (also called \emph{graphed} in \cite[Definition 4.3]{arxiv1902.09375}) if the source and range maps of $\Lambda$ are constant on $\rho$-equivalence classes.

Then the following statements are equivalent (assuming that $\Lambda$ has no sources nor sinks):
\begin{enumerate}[label=(\arabic*)]
    \item $\pi\colon\Lambda\to\Gamma$ is a rigid homomorphism;
    \item The vertex map $\pi^{(0)}\colon\Lambda^{(0)}\to\Gamma^{(0)}$ is injective;
    \item $\ker\pi$ is a rigid congruence.
\end{enumerate}

This will be the additional condition for the semigroupoid homomorphisms we consider for $R$-bundles.

\begin{definition}\label{def:R.bundle}
    An \emph{$R$-bundle} consists of a rigid semigroupoid homomorphism $\pi\colon\Lambda\to\Gamma$, together with an $R$-bimodule structure on the fiber $\pi^{-1}(\gamma)$ for each $\gamma\in\Gamma$, such that for every $(\gamma_1,\gamma_2)\in\Gamma^{(2)}$, the product map
    \[\mu|_{\pi^{-1}(\gamma_1)\times\pi^{-1}(\gamma_2)}\colon\pi^{-1}(\gamma_1)\times\pi^{-1}(\gamma_2)\to\pi^{-1}(\gamma_1\gamma_2),\qquad (x,y)\mapsto xy\]
    is $R$-balanced; thus it is regarded as an $R$-bimodule homomorphism $\mu_{(\gamma_1,\gamma_2)}\colon \pi^{-1}(\gamma_1)\tensor[_R]{\otimes}{_R}\pi^{-1}(\gamma_2)\to\pi^{-1}(\gamma_1\gamma_2)$.
    
    We will refer simply to $\pi$ as the $R$-bundle. The zero of $\pi^{-1}(\gamma)$ may be denoted by $0_\gamma$ if necessary, or simply $0$.
    
    If $\Lambda$ and $\Gamma$ are topological semigroupoids, the $R$-bundle $\pi\colon\Lambda\to\Gamma$ is said to be \emph{continuous} if
    \begin{enumerate}[label=(\roman*)]
        \item\label{def:R.bundle.cont} $\pi$ is continuous;
        \item\label{def:R.bundle.add.cont} The addition $+$ is continuous from $\left\{(x,y)\in\Lambda\times\Lambda:\pi(x)=\pi(y)\right\}$ to $\Lambda$;
        \item\label{def:R.bundle.bimod.cont} Left and right scalar multiplications are continuous (i.e., the map $(r,x)\mapsto(rx,xr)$ from $R\times\Lambda$ to $\Lambda\times\Lambda$ is continuous).
        \item\label{def:R.bundle.0.cont} The zero function $\mathbf{0}\colon\Gamma\to\Lambda$, $\gamma\mapsto 0_\gamma$ is continuous;
    \end{enumerate}
\end{definition}

Note that the nontrivial inclusion $(\pi\times\pi)^{-1}(\Gamma^{(2)})\subseteq\Lambda^{(2)}$, which is guaranteed as $\pi$ is a rigid homomorphism, implies that $\pi^{-1}(\gamma_1)\times\pi^{-1}(\gamma_2)$ is contained in $\Lambda^{(2)}$ whenever $(\gamma_1,\gamma_2)\in\Gamma^{(2)}$, so $\mu_{(\gamma_1,\gamma_2)}$ is a well-defined map.

Moreover, since bimodules are nonempty by definition, then $\pi^{-1}(\gamma)$ is nonempty for all $\gamma\in\Gamma$, so $\pi$ is surjective.

\begin{example}\label{ex:0.discontinuous}
    The assumption that the zero function is continuous does not automatically follow from the other ones, and it will be necessary in the proof of Theorem \ref{thm:isomorphism.of.naive.crossed.product}. For example, let $R$ be any unital topological ring, and $\Gamma$ be any topological semigroupoid without isolated points. Let $\Lambda$ be the same semigroupoid as $\Gamma$, however with the discrete topology. Let $I\colon\Lambda\to\Gamma$ be the identity map, and regard each fiber $I^{-1}(\gamma)=\left\{\gamma\right\}$ as the zero $R$-module, so that we have an $R$-bundle.
    
    Then the zero function $\mathbf{0}\colon\Gamma\to\Lambda$ is the identity map, and it is discontinuous everywhere, even though conditions \ref{def:R.bundle.cont}-\ref{def:R.bundle.bimod.cont} are satisfied.
\end{example}

\begin{remark}
    Suppose that $\pi\colon\Lambda\to\Gamma$ is a continuous $R$-bundle. If $X$ is a topological space, $f\colon X\to\Lambda$ and $h\colon X\to R$ are continuous functions, then the function $hf\colon x\mapsto h(x)f(x)$ is continuous, as it is the composition of $(h,f)\colon X\to R\times \Lambda$ with the scalar multiplication, and similarly $fh$ is continuous.
\end{remark}

\subsection{Sectional algebras}

We will now define the sectional algebra of an $R$-bundle, where $R$ is a unital topological ring. As we work in a general setting, we will need to assume at least that a version of Urysohn's Lemma holds for $R$-valued functions.

\begin{definition}
    Let $R$ be a unital topological ring. A topological space $X$ is said to be \emph{$R$-normal} if for any two disjoint closed subsets $A,B\subseteq X$ there exists a continuous function $f\colon X\to R$ such that $f=0$ on $A$ and $f=1$ on $B$. The space $X$ is said to be \emph{locally $R$-normal} if every point of $X$ admits a basis of $R$-normal neighbourhoods.
\end{definition}

Note that every closed subset of an $R$-normal space is again $R$-normal, with the subspace topology.

This definition covers all cases of interest to us. For example, every locally compact, locally Hausdorff space is locally $\mathbb{R}$ and locally $\mathbb{C}$-normal. More generally, if $R$ is a path-connected unital topological ring (e.g. a unital Banach algebra), then every locally compact, locally Hausdorff space is locally $R$-normal. 

In the more extreme case, every locally compact, locally Hausdorff, zero-dimensional space is locally $R$-normal for any unital topological ring $R$. In particular, ample groupoids (as in \cite[Definition 2.2.4]{MR1724106}) are $R$-normal for any unital topological ring $R$.

For the next definitions, we fix a continuous $R$-bundle $\pi\colon\Lambda\to\Gamma$, where $\Gamma$ is étale. If $X$ is any topological space, the \emph{support} of a function $f\colon X\to Y$ is the closure of $\left\{x\in X:f(x)=0_{\pi(f(x))}\right\}$ in $X$.

\begin{definition}
    A \emph{section} of $\pi$ is a right-inverse of $\pi$, i.e., a function $\alpha\colon\Gamma\to\Lambda$ such that $\pi\circ \alpha=\id_{\Gamma}$. We consider the set of sections of $\pi$ as an $R$-bimodule under pointwise addition and product by $R$.
    
    Given an open subset $V\subseteq\Gamma$, let $C_c(V,\pi)$ be the collection of all sections $\alpha\colon\Gamma\to\Lambda$ such that $\alpha=0$ outside $V$, $\alpha$ is continuous on $V$, and $\supp(\alpha)\cap V$ is compact.
\end{definition}

If $V$ is an open Hausdorff subset of $\Gamma$ and $\alpha$ is a section of $\pi$, note that $\supp(\alpha)\cap V$ is simply the support of the restriction of $\alpha$ to $V$. Since $V$ is Hausdorff, then $\supp(\alpha)\cap V$ is compact if and only if $\alpha=0$ outside of a compact subset of $V$. Thus we shall say that elements of $C_c(V,\pi)$ are \emph{compactly supported in $V$}.

\begin{definition}\label{def:sectional.algebra}
    The \emph{sectional algebra} of the continuous $R$-bundle $\pi\colon\Lambda\to\Gamma$ is the $R$-bimodule $\mathcal{A}(\pi)$ generated by the union of $C_c(V,\pi)$ for all Hausdorff open subsets $V\subseteq\Gamma$.
\end{definition}

By the remark after Definition \ref{ex:0.discontinuous}, $\mathcal{A}(\pi)$ is a $C(\Gamma,R)$-bimodule with pointwise product, i.e., if $f\colon\Gamma\to R$ is continuous and $\alpha\in\mathcal{A}(\pi)$, then the function $f\alpha\colon \gamma\mapsto f(\gamma)\alpha(\gamma)$ also belongs to $\mathcal{A}(\pi)$ (and similarly $\alpha f\in\mathcal{A}(\pi)$).

In order to define the product structure of $\mathcal{A}(\pi)$, we will need to make use of the fact that $\Gamma$ is étale (as we specified above). In this case, the first point is to prove that we may strengthen the condition in Definition \ref{def:sectional.algebra}, to allow us to take generating elements of $\mathcal{A}(\pi)$ in sets of the form $C_c(V,\pi)$ where $V$ belongs to some prescribed basis of $\Gamma$.

The following Lemma may be proven, with obvious modifications, as in the classical ($\mathbb{R}$-valued) case, see e.g. \cite[Theorem 2.13]{MR924157}.

\begin{lemma}[Existence of partitions of unity]\label{lem:partitionsofunity}
    Suppose that $R$ is a unital topological ring and $X$ is a Hausdorff, locally compact, locally $R$-normal space. Then for every compact $K\subseteq X$ and every finite open cover $V_1,\ldots,V_n$ of $K$, there exist continuous functions $f_1,\ldots,f_n\colon X\to R$ such that $\supp(f_i)\subseteq V_i$ for all $i$, and $\sum_{i=1}^n f_i=1$ on $K$. Moreover, $f_i$ may be taken to have compact support on $X$.
\end{lemma}

As usual, we call the collection $f_1,\ldots,f_n$ a \emph{partition of unity subordinate to $V_1,\ldots,V_n$}.

In fact, the same technique as in \cite[Theorem 2.13]{MR924157} gives us the useful permanence of the $R$-normal property in the case of interest. 

\begin{lemma}
    Let $R$ be a unital topological ring and $X$ and $Y$ two compact and Hausdorff spaces. If both $X$ and $Y$ are $R$-normal, then $X\times Y$ is also $R$-normal.
    
    In particular, any finite product of locally compact, locally Hausdorff, and locally $R$-normal spaces is locally $R$-normal.
\end{lemma}

The second statement above follows from the fact that any locally compact, locally Hausdorff and locally $R$-normal space $X$ admits a neighbourhood basis of sets which are simultaneously Hausdorff, compact and $R$-normal. Indeed, for every $x\in X$ and every open set $U$ containing $X$, there exist neighbourhoods $K$, $N$ and $H$ of $x$ such that $K$ is compact, $N$ is $R$-normal, and $H$ is Hausdorff, and $x\in K\subseteq N\subseteq H\subseteq U$. As $H$ is Hausdorff and $K$ is compact, $K$ is closed in $H$, so it is also closed in $N$, and thus it is $R$-normal, compact and Hausdorff.

\begin{lemma}\label{lem:sectionalalgebraisgeneratedbybasis}
    Suppose that $R$ is a unital topological ring, and $\pi\colon\Lambda\to\Gamma$ is a continuous $R$-bundle where $\Gamma$ is an étale, locally $R$-normal semigroupoid. If $\mathscr{B}$ is any topological basis of $\Gamma$, then $\mathcal{A}(\pi)$ is generated, as an additive group, by $\bigcup\left\{C_c(U,\pi):U\in\mathscr{B}\right\}$.
\end{lemma}
\begin{proof}
    It is sufficient to prove that if $V$ is open in $\Gamma$ and Hausdorff, and $\alpha\in C_c(V,\pi)$, then $f$ is a sum of sections in $\bigcup\left\{C_c(U,\pi):U\in\mathscr{B}\right\}$.
    
    Consider a cover $B_1,\ldots,B_n$ of the compact $\supp(\alpha)\cap V$, where $B_i\in\mathscr{B}$ and $B_i\subseteq V$. Take a partition of unity $h_1,\ldots,h_n$ of $\supp(\alpha)\cap V$ subordinate to $B_1,\ldots,B_n$. Then for each $i$, the section $\alpha_i\defeq h_i\alpha$ is continuous and compactly supported on $B_i$, i.e., $\alpha_i\in C_c(B_i,\pi)$. Since $\sum_{i=1}^n h_i=1$ on $\supp(\alpha)\cap V$ then $\sum_{i=1}^n\alpha_i=\alpha$.\qedhere
\end{proof}

Let us reiterate that $\Gamma$ is étale. As a last point before being able to describe the multiplicative structure of $\mathcal{A}(\pi)$, note that for all $\alpha\in\mathcal{A}(\pi)$ and all $x\in\Gamma^{(0)}$, there are only finitely many elements $\gamma\in\Gamma$ such that $\so(\gamma)=x$ and $\alpha(\gamma)\neq 0$: Indeed, there exists a compact set $K$ such that $\alpha=0$ outside $K$. As $\Gamma$ is étale, then $\so^{-1}(x)$ is a closed subspace of $\Gamma$, and discrete with the subspace topology. In particular $\so^{-1}(x)\cap K$ is a discrete, closed subset of the compact $K$, and hence it is finite, and all elements $\gamma\in\so^{-1}(x)$ for which $\alpha(\gamma)\neq 0$ belong to this set.

Similarly, there are only finitely many $\gamma\in\Gamma$ such that $\ra(\gamma)=x$ and $\alpha(\gamma)\neq 0$

We thus define the \emph{convolution product} of two sections $\alpha,\beta\in \mathcal{A}(\pi)$ as
\[(\alpha\ast\beta)(\gamma)=\sum_{ab=\gamma}\alpha(a)\beta(b).\]
Note that if $ab=\gamma$ in $\Gamma$, then $\ra(a)=\ra(\gamma)$ and $\so(b)=\so(\gamma)$, so the sum above is finite by the previous paragraph. Then $\alpha\ast\beta$ is a well-defined section of $\pi$. However, we still need to verify that it is an element of $\mathcal{A}(\pi)$.

\begin{proposition}
    Suppose that $\Gamma$ is étale and locally $R$-normal. If $\alpha,\beta\in \mathcal{A}(\pi)$ then $\alpha\ast\beta\in \mathcal{A}(\pi)$.
\end{proposition}
\begin{proof}
    By Lemma \ref{lem:sectionalalgebraisgeneratedbybasis}, it is enough to prove that if $\alpha\in C_c(U,\pi)$ and $\beta\in C_c(V,\pi)$, where $U,V\in\mathbf{B}(\Gamma)$ are open bisections, then $\alpha\ast\beta\in C_c(UV,\pi)$. For this, consider the compact sets $K\defeq \supp(\alpha)\cap U$ and $L\defeq\supp(\beta)\cap V$
    
    By definition of the convolution product, $\alpha\ast\beta=0$ outside of $KL$, which is compact (Proposition \ref{prop:product.map.is.local.homeo}), and in particular $\alpha\ast\beta=0$ outside $UV$. So in order to conclude that $\alpha\ast\beta\in C_c(UV,\pi)$, we just need to prove that $\alpha\ast\beta$ is continuous on $UV$.
    
    Suppose that we have a converging net $\gamma_i\to \gamma$ in $UV$. First, write $\gamma_i=u_iv_i$ for each $i$, and $a=uv$, where $u_i,u\in U$ and $v_i,v\in V$. Since $U$ and $V$ are bisections and $\alpha$ and $\beta$ are zero outside $U$ and $V$, respectively, then $(\alpha\ast\beta)(\gamma_i)=\alpha(u_i)\beta(v_i)$ for each $i$, and $(\alpha\ast\beta)(\gamma)=\alpha(u)\beta(v)$. Since $\gamma_i\to \gamma$, then $\ra(u_i)=\ra(\gamma_i)\to\ra(\gamma)=\ra(u)$, and since the range map is a homeomorphism from $U$ to $\ra(U)$ then $u_i\to u$, so $\alpha(u_i)\to \alpha(u)$ as $\alpha$ is continuous on $U$. Similarly, $\beta(v_i)\to\beta(v)$, and we conclude that $(\alpha\ast\beta)(\gamma_i)=\alpha(u_i)\beta(v_i)\to \alpha(u)\beta(v)=(\alpha\ast\beta)(\gamma)$, since the product of $\Lambda$ is continuous. Therefore $\alpha\ast\beta$ is continuous on $UV$.\qedhere
\end{proof}

\subsection{Graded sectional algebras}

In the next section, we will describe several classes of algebras as sectional algebras in canonical manners. In particular, all graded algebras may be regarded as sectional algebras of bundles over discrete semigroupoids. We will thus be interested in considering \emph{graded} sectional algebras as well. This is in vogue with the current trend on the study of graded Steinberg algebras, and more specifically Leavitt path algebras. See \cite{MR2310414,MR3848066,MR3372123,MR2363133,MR2738365}. Graded algebras and their graded homomorphisms are, by definition, more rigid (and manageable) than general ungraded algebras, but their graded theory alone may be used to obtain insight and in fact results about the ungraded structure as well. See \cite{MR2046303} for an introduction to the classical aspects of the subject, \cite{MR1632749,MR1895414,MR3339401} for semigroup graded rings, and \cite{MR2111220} for groupoid graded rings.

The same notion of grading by a category (e.g.\ see \cite[Definition 2.1]{MR2111220}) may be used in the context of semigroupoids, so as to cover both semigroup and category gradings.

\begin{definition}\label{def:gradedalgebraoversemigroupoid}
    Let $G$ be a discrete semigroupoid. An algebra $A$ over a ring $R$ is \emph{$G$-graded} if it is equipped with a family $\left\{A_g\right\}_{g\in G}$ of sub-bimodules satisfying
\begin{enumerate}[label=(\roman*)]
    \item $A=\oplus_{g\in G}A_g$;
    \item $A_gA_h\subseteq A_{gh}$ whenever $(g,h)\in G^{(2)}$;
    \item $A_gA_h=\left\{0\right\}$ whenever $(g,h)\not\in G^{(2)}$.
\end{enumerate}
    Each sub-bimodule $A_g$ is called a \emph{homogeneous component} of the graded algebra $A$.
    
    A \emph{graded homomorphism} between two $G$-graded algebras $A=\oplus_{g\in G}A_g$ and $B=\oplus_{g\in G}B_g$ is an algebra homomorphism $\phi\colon A\to B$ satisfying $\phi(A_g)\subseteq B_g$ for all $g\in G$. If $\phi$ is bijective then it is called a \emph{graded isomorphism}. (The inverse of any graded isomorphism is also a graded homomorphism.)
\end{definition}

Suppose that $\pi\colon\Lambda\to\Gamma$ is a continuous $R$-bundle, where $\Gamma$ is étale and locally $R$-normal. $G$ is a discrete semigroupoid and $c\colon \Gamma\to G$ is a continuous semigroupoid homomorphism. We also call the homomorphism $c$ a \emph{grading} of $\Gamma$.

Define a grading on $\mathcal{A}(\pi)$ by setting, for all $g\in G$
\[\mathcal{A}(\pi)_g=\sum\left\{C_c(U,\pi):U\text{ is open, Hausdorff and }U\subseteq c^{-1}(g)\right\}.\]
Note that, as $G$ is discrete and $c$ is continuous, then $c^{-1}(g)$ is open in $\Gamma$ for all $g\in G$. By Lemma \ref{lem:sectionalalgebraisgeneratedbybasis}, $\mathcal{A}(\pi)$ is generated as an additive group by $\bigcup_{g\in G}A(\pi)_g$. It is then straightforward to verify that this yields a $G$-graded structure to $\mathcal{A}(\pi)$. With this $G$-grading, we call $\mathcal{A}(\pi)$ a \emph{graded sectional algebra} (via the homomorphism $c$).

The following alternative description of the homogeneous components $\mathcal{A}(\pi)_g$ is useful:

\begin{lemma}
Given $g\in G$, the following two equalities hold:
\[\mathcal{A}(\pi)_g=\left\{\alpha\in\mathcal{A}(\pi):\alpha=0\text{ outside of }c^{-1}(g)\right\}=\left\{\alpha\in\mathcal{A}(\pi):\supp(\alpha)\subseteq c^{-1}(g)\right\}.\]
\end{lemma}
\begin{proof}
Since $c\colon\Gamma\to G$ is continuous and $G$ is discrete, then $c^{-1}(g)$ is clopen, so the rightmost equality holds. Moreover, every element of $\mathcal{A}(\pi)_g$ is clearly zero outsize $c^{-1}(g)$.

It is thus enough to prove that any $\alpha\in\mathcal{A}(\pi)$ which is zero outsize $c^{-1}(g)$ belongs to $\mathcal{A}(\pi)_g$. Indeed, write $\alpha=\sum_{i=1}^n\alpha_i$, where $\alpha_i\in C_c(V_i,\pi)$ for certain Hausdorff open subsets $V_i\subseteq\Gamma$. Consider the ($R$-valued) characteristic function $1_{c^{-1}(g)}$ of $c^{-1}(g)$. Since $\alpha=0$ outside $c^{-1}(g)$ then \[\alpha=1_{c^{-1}(g)}\alpha=\sum_{i=1}^n 1_{c^{-1}(g)}\alpha_i,\]
where the products are pointwise. For each $i=1,\ldots,n$, we have $1_{c^{-1}(g)}\alpha_i\in C_c(V_i\cap c^{-1}(g),\pi)$, again because $c^{-1}(g)$ is clopen.

Therefore $1_{c^{-1}(g)}\alpha_i\in\mathcal{A}(\pi)_g$ for each $i$, so $\alpha\in\mathcal{A}(\pi)_g$ as well.\qedhere
\end{proof}

If $\Gamma$ is a discrete semigroupoid, then we may grade $\mathcal{A}(\pi)$ via the identity morphism $\id_\Gamma$ of $\Gamma$. We call this the \emph{trivial grading} of $\mathcal{A}(\pi)$. In this case, the homogeneous component at an element $\gamma\in\Gamma$ is isomorphic to $\pi^{-1}(\gamma)$; Namely, the map
\[\mathcal{A}(\pi)_\gamma\to\pi^{-1}(\gamma),\qquad \alpha\mapsto\alpha(\gamma)\]
is a bimodule isomorphism.

\section{Examples of sectional algebras}

\subsection{Graded algebras as sectional algebras}

As explained in the Introduction, continuous $R$-bundles may be interpreted as ``continuous gradings'' of $R$-algebras. We may make this interpretation formal by considering the inverse direction, and showing that bundles over discrete semigroupoids are equivalent to graded algebras.

In one direction, we already know how to associate a graded algebra to any discrete $R$-bundle $\pi\colon\Lambda\to G$ (i.e., both $\Lambda$ and $G$ are discrete): Simply take the sectional algebra $\mathcal{A}(\pi)$ with its trivial $G$-grading.

In the other direction, given a $\mathcal{G}$-graded algebra $A=\oplus_{g\in\Gamma}A_g$ we construct the semigroupoid $S(A)=\bigcup_{g\in\mathcal{G}}\left\{g\right\}\times A_g$ (which may be regarded as the disjoint union of the homogeneous components of $A$). The product of $S(A)$ is determined by
\[(g,a)(h,b)=(gh,ab)\]
whenever $(g,h)\in\mathcal{G}^{(2)}$. The first coordinate projection $\eta_A\colon S(A)\to\mathcal{G}$, $(g,a)\mapsto g$, is a rigid homomorphism, and each section $\eta_A^{-1}(g)=\left\{g\right\}\times A_g$ has an $R$-bimodule structure coming from $A_g$. This determines an algebraic bundle $\eta_A$.

These two constructions are inverse to each other in a categorical sense: if we start with a graded algebra $A=\oplus_{g\in G}A_g$, then $\mathcal{A}(\eta_A)$ and $A$ are graded isomorphic: If we denote by $\pi_2(s,a)=a$ the second coordinate projection of $S(A)$, then the map  \[\mathcal{A}(\eta_A)\to A,\qquad \alpha\mapsto\sum_{g\in G}\pi_2(\alpha(g))\]
is a graded isomorphism.

Similarly, it is also straightforward to verify that if $\pi\colon\Lambda\to G$ is a discrete $R$-bundle then $\Lambda$ is isomorphic, as a semigroupoid, to $S(\mathcal{A}(\pi))$. The isomorphism is obtained by regarding $S(\mathcal{A}(\pi))$ as the disjoint union of the homogeneous components $\mathcal{A}(\pi)_g$, which we already know to be isomorphic to $\pi^{-1}(g)$. So we have an identification $S(\mathcal{A}(\pi))\simeq \sqcup_{g\in G}\pi^{-1}(g)=\Lambda$.

To completely formalize this equivalence, let us say that a homomorphism between bundles $\pi^i\colon\Lambda^i\to G$, $i=1,2$ (over the same semigroupoid $G$) is a semigroupoid homomorphism $\phi\colon\Lambda^1\to\Lambda^2$ for which $\pi^2\circ\phi=\pi^1$. Straightforward arguments, much as above, show that the category of $G$-graded algebras and their graded homomorphisms is equivalent to the category of discrete bundles over $G$ and their homomorphisms.

\subsection{``Naïve'' Crossed products as sectional algebras}

Crossed products of $C^*$-algebras by actions of inverse semigroups were originally considered by Sieben in \cite{MR1456588}, as an alternative to Exel's approach to $C^*$-dynamics via partial group actions (\cite{MR1276163}). Sieben's main result was that every $C^*$-crossed product by a partial group action is isomorphic to a $C^*$-crossed product by a (global) inverse semigroup action. However, $C^*$-crossed products by inverse semigroups were defined in terms of ``covariant representations'' in \cite{MR1456588}. 

In \cite{MR2559043}, Exel gave an alternative description of inverse semigroup crossed products, more algebraic in flavour and may be applied as well in the discrete setting (for example, when considering Steinberg algebras). Let us briefly and somewhat informally describe the procedure to construct an inverse semigroup crossed product: Let $\theta$ be a global action of an inverse semigroup $S$ on an algebra $A$:
\begin{enumerate}[label=\arabic*.]
    \item First, one proceeds in a manner similar as to when constructing a twisted group algebra: Consider the bimodule of all finite sums of elements $\delta_s a$, where $a\in\dom(\theta_s)$, with product determined by $(\delta_s a)(\delta_t b)=\delta_{st}\theta_{t^*}(a\theta_t(b))$. This defines an algebra $S\bigstar A$, which we call the ``naïve'' crossed product.
    \item Then consider the ideal $N$ of $S\bigstar A$ generated by all terms of the form $\delta_s a-\delta_t a$, where $s\leq t$ and $a\in\dom(\theta_s)$. The crossed product of $A$ by $S$ is the quotient of $S\bigstar A$ by $N$. (In the case of actions of groups this step is unnecessary.)
\end{enumerate}

In the $C^*$-algebraic case one then takes the $C^*$-envelope of the resulting algebra $S\bigstar A/N$. For further and more general reference, see \cite{MR3231479} in the $C^*$-algebraic case, and \cite{arxiv1804.00396,MR3943326} in the discrete case.

In this subsection we will deal with the first step described above, while the second one will be considered in Subsection \ref{subsec:quotients}. More precisely, we will now describe naïve crossed products as sectional algebras.

Let $A$ be a topological $R$-algebra, where $R$ is a unital topological ring. By a $\land$-preaction (or partial action, or global action) $\theta$ of an étale inverse semigroupoid $S$ on $A$, as an algebra, we shall mean a $\land$-preaction of $S$ on $A$, as in Definition \ref{def:action}, where $A$ is regarded as a semigroupoid under product, which also satisfies:
\begin{itemize}
    \item For every $s\in\mathcal{S}$, $\dom(\theta_s)$ is a sub-bimodule of $A$ and $\theta_s$ is a bimodule isomorphism.
\end{itemize}

Note that we already assume that, for every $v\in\mathcal{S}
$, $I(\theta,v)=\bigcup_{s\in\so^{-1}(v)}\dom(\theta_s)$ is a multiplicative ideal of $A$, and hence it is an ideal as an $R$-algebra. Similarly, $\dom(\theta_s)$ is a multiplicative ideal of $I(\theta,\so(s))$ for each $s\in\mathcal{S}$, so $\dom(\theta_s)$ is a subalgebra of $A$. As $\theta_s$ preserves products (being a semigroupoid homomorphism), it is actually an algebra isomorphism.

Suppose also that $\theta$ is continuous and associative. Consider the semidirect product semigroupoid $\mathcal{S}\ltimes A$ (not to be confused with crossed products of \cite{arxiv1804.00396,MR3943326,MR3231479}). The product is given by
\[(s,a)(t,b)=(st,\theta_{t^*}(a\theta_t(b))\qquad\text{whenever sensible}.\]
Let $\pi\colon\mathcal{S}\ltimes A\to\mathcal{S}$ be the projection $\pi(s,a)=s$. Then each section $\pi^{-1}(s)=\left\{s\right\}\times\dom(\theta_s)$ has an obvious $R$-bimodule structure coming from $\dom(\theta_s)$. In this manner, $\pi\colon \mathcal{S}\ltimes A\to\mathcal{S}$ is a continuous algebraic bundle. Since we assume that $\mathcal{S}$ is étale, we may consider the respective sectional algebra.
    
\begin{definition}
    The \emph{naïve crossed product} (induced by $\theta$) is the sectional algebra $S\bigstar A\defeq \mathcal{A}(\pi)$.
\end{definition}

We may simplify the description of the crossed product to a more usual approach as follows: A section $\alpha\colon\mathcal{S}\to\mathcal{S}\ltimes A$ of $\pi$ always has the form $\alpha=(\id_{\mathcal{S}},f)$ for some unique function $f\colon\mathcal{S}\to A$ satisfying $f(s)\in\dom(\theta_s)$ for all $s\in\mathcal{S}$. Moreover, the sets of continuity points of $\alpha$ and of $f$ coincide, as do the supports of $\alpha$ and of $f$.

Therefore, we may instead regard elements $\mathcal{S}\bigstar A$ as functions from $\mathcal{S}$ to $A$, and obtain the alternative description (up to natural isomorphism):

\begin{definition}
    The \emph{naïve crossed product} $\mathcal{S}\bigstar A$ induced by a continuous, associative $\land$-preaction $\theta$ of an étale, locally $R$-normal inverse semigroupoid $\mathcal{S}$ on a topological $R$-algebra $A$ is the $R$-algebra generated by all functions $f\colon\mathcal{S}\to A$ such that
    \begin{enumerate}[label=(\roman*)]\label{def:alternative.definition.for.naïve.crossed.product}
        \item\label{def:alternative.definition.for.naïve.crossed.product1} for all $s\in\mathcal{S}$, $f(s)\in\dom(\theta_s)$;
        \item\label{def:alternative.definition.for.naïve.crossed.product2} There exists an open and Hausdorff subset $V\subseteq \mathcal{S}$ such that $f=0$ outside $V$, and $\supp(f)\cap V$ is compact.
    \end{enumerate}
    As in Lemma \ref{lem:sectionalalgebraisgeneratedbybasis}, we may restrict the sets in \ref{def:alternative.definition.for.naïve.crossed.product2} to belonging to any basis of $\mathcal{S}$.
    
    The $R$-bimodule structure of $\mathcal{S}$ is the pointwise one, whereas the convolution product is given by
    \[(f\ast g)(s)=\sum_{xy=s}(\theta_{y^*}(f(x)\theta_y(g(y))).\]
    
    If $c\colon\mathcal{S}\to G$ is a continuous homomorphism from $\mathcal{S}$ to a discrete semigroupoid $G$, then $S\bigstar A$ is graded, with homogeneous component $(S\bigstar A)_g$ the set of function $f\colon\mathcal{S}\to A$ as above which vanish outside of $c^{-1}(g)$.
\end{definition}

\begin{remark}
    We use the convention that the elements of $\mathcal{S}$ should be thought of as functions, and thus act on the left on elements in their domains, so in some sense a function $f$ in a crossed product should be regarded as a ``continuous sum'' $f=\sum\delta_s f(s)$, where $\delta_s$ denotes the Dirac function at $s\in\mathcal{S}$. Thus for such an interpretation to be valid we need that $f(s)\in\dom(\theta_s)$.
    
    The reverse approach is more common: Regard $f$ as a continuous sum $f=\sum f(s)\delta_s$, so we should instead require that $f(s)\in\ran(\theta_s)$. This is the approach taken in \cite[p.\ 79]{MR1331978} for partial actions of discrete groups on $C^*$-algebras; on \cite[p.\ 9]{MR1456588} for actions of inverse semigroups on $C^*$-algebras; and \cite[Definition 2.5]{MR3743184} for partial actions of inverse semigroups on discrete algebras.
    
    Both of these approaches are equivalent, since we may use the $\land$-preaction $\theta$ itself to move elements from $\dom(\theta_s)$ to $\ran(\theta_s)$ and vice-versa. 
    
    More precisely, let us denote, just as in \cite[p.\ 3]{arxiv1804.00396}, by $\mathscr{L}(\theta)$ the span (as an $R$-bimodule) of all functions $f\colon\mathcal{S}\to A$, which are continuously and compactly supported on some open bisection of $\mathcal{S}$, and which satisfy $f(s)\in\ran(\theta_s)$ for all $s\in\mathcal{S}$, with pointwise $R$-bimodule structure, and product given by
    \[(f\ast g)(s)=\sum_{xy=s}\theta_x\left(\theta_{x^*}(f(x))g(y)\right).\]
    
    Then the map \[\phi\colon S\bigstar A\to\mathscr{L}(\theta),\qquad  \phi(f)(s)=\theta_s(f(s)),\]
    is promptly verified to be an algebra isomorphism. (We remark that it is necessary to use Equation \eqref{eq:landpreactionchangebygreater} to check that $\phi$ is multiplicative.)
\end{remark}

\subsection{Semigroupoid algebras as sectional algebras}\label{subsectiion:semigroupoidalgebra}

Here we give a simple adaptation of the usual notions of ``category algebras'' and ``semigroup algebra'' to the context of étale semigroupoids. The definition we use is based on that of \emph{Steinberg algebras} of ample groupoids, which were first considered in \cite{MR2565546} and \cite{MR3274831}, as ``algebraizations'' of groupoid $C^*$-algebras and as models for Leavitt path algebras, and also work as a ``laboratory'' (\cite{arxiv1806.04362}) for studying groupoid $C^*$-algebras.

Let $R$ be a unital topological ring, $A$ a topological $R$-algebra, and $\Gamma$ an étale, locally $R$-normal semigroupoid.

\begin{definition}
    The \emph{semigroupoid algebra} $A\Gamma$ is the algebra generated by all functions $f\colon\Gamma\to A$ for which there exists an open Hausdorff $V\subseteq\Lambda$ such that
\begin{enumerate}[label=(\roman*)]
    \item\label{def:steinbergalgebragenerator1} $f=0$ outside $V$;
    \item\label{def:steinbergalgebragenerator2} $f|_V\colon V\to A$ is continuous and compactly supported.
\end{enumerate}
    The bimodule structure of $A\Gamma$ is pointwise, and the product of $f,g\in A\Gamma$ is $(f\ast g)(\gamma)=\sum_{ab=\gamma}f(a)g(b)$.
    
    Just as in Lemma \ref{lem:sectionalalgebraisgeneratedbybasis}, it is sufficient to consider the open sets $V$ of \ref{def:steinbergalgebragenerator2} as belonging to a prescribed basis of $\Gamma$.
\end{definition}

If $\mathcal{S}$ is an étale, locally $R$-normal inverse semigroupoid, then the algebra $A\mathcal{S}$ is simply the crossed product under the trivial action of $\mathcal{S}$ on $A$: $\theta_s(a)=a$ for all $s\in\mathcal{S}$ and $a\in A$, and thus may be realized as a sectional algebra as in the previous subsection. The same is true for non-inverse semigroupoids, as we detail below.

Let $R$, $A$ and $\Gamma$ be as above. Consider the topological semigroupoid $A\times\Gamma$, with entrywise product $(a,\gamma)(b,\delta)=(ab,\gamma\delta)$ whenever $\gamma\delta$ is defined in $\Gamma$. Consider the projection $\pi\colon A\times\Gamma\to\Gamma$, $\pi(a,\gamma)=\gamma$. For every $\gamma\in\Gamma$, the preimage $\pi^{-1}(\gamma)=A\times\left\{\gamma\right\}$ has an obvious $R$-bimodule structure induced by $A$, and in this way, $\pi$ is an algebraic bundle. Then $A\Gamma$ is isomorphic to the sectional algebra $\mathcal{A}(\pi)$, just as in the previous subsection.

If $c\colon\Gamma\to G$ is a grading of $\Gamma$ over a discrete semigroupoid $G$, then the groupoid algebra $A\Gamma$ is graded with homogeneous components $(A\Gamma)_g$ consisting of all functions $f\in A\Gamma$ which vanish outside of $c^{-1}(g)$.

\section{\texorpdfstring{The compact-open topology of $\mathcal{A}(\pi)$}{The compact-open topology of A(π)}}

    Since we consider algebraic structures with compatible topologies, it is natural to ask whether the construction of the sectional algebras ends up in that same category, i.e., if we can make $A(\pi)$ into a topological algebra in a natural manner. In fact, this will be an issue that we will need to consider in Theorem \ref{thm:isomorphism.of.naive.crossed.product}.
    
    The study of topologies on spaces of functions is a classical problem of elementary topology, and is motivated by the question of exponentiability of a topological space $X$: A topological space is \emph{exponentiable} if for any other topological space $Y$, we may topologize the set $C(X,Y)$ of continuous functions from $X$ to $Y$ in such a manner that for any topological space $A$, the sets $C(A\times X,Y)$ and $C(A,C(X,Y))$ are in natural bijection. (This is the same as exponentiability of $X$ in the category of topological spaces.) The exponentiable spaces are precisely the \emph{core compact} ones, which in particular include all locally compact Hausdorff spaces. In this case, the topology on $C(X,Y)$ is the \emph{compact-open topology}, also known as the \emph{topology of compact convergence}. See \cite{MR2032837} (specially Theorem 5.3) for reference.
    
    \begin{definition}
        Let $X$ and $Y$ be topological spaces and let $\mathcal{A}$ be any collection of functions from $X$ to $Y$. The \emph{compact-open topology} on $\mathcal{A}$ is the topology generated by sets of the form
        \[[K,V]=\left\{f\in\mathcal{A}:f(K)\subseteq V\right\},\]
        where $K\subseteq X$ is compact and $V\subseteq Y$ is open. Note that these sets form a sub-basis for the compact-open topology, and not a basis.
    \end{definition}
    
    However, the compact-open topology does not make the sectional algebra $\mathcal{A}(\pi)$ of a continuous $R$-bundle into a topological algebra, even in the discrete (and in particular Hausdorff) setting.
    
    \begin{example}
        Consider the action $\alpha$ of $\mathbb{Z}$ on itself by addition: $\alpha_m(n)=m+n$, for all $m,n\in\mathbb{Z}$. Let $\Gamma=\mathbb{Z}\ltimes\mathbb{Z}$ be the (discrete) semidirect product groupoid. The product is given by $(m,n+p)(n,p)=(m+n,p)$. Let $R$ be any nontrivial unital discrete ring. We regard the groupoid algebra $R\Gamma$ as a sectional algebra as in Subsection \ref{subsectiion:semigroupoidalgebra}.
        
        Then the compact-open topology of $R\Gamma$ coincides with the topology of pointwise convergence. The sequence of functions $f_k=\delta_{(-k,k)}+\delta_{(k,0)}$, where $\delta_\ast$ denotes the Kronecker delta, converges to $0$, but $(f_kf_k)(0,0)=1$ for all $k$, so the product of $R\Gamma$ is not continuous.
    \end{example}
    
    In the example above, we used the fact that there are sequences of elements in $\Gamma^{(2)}$ converging to infinity, but whose product always lie in a prescribed compact set (namely, $(-k,k)(k,0)=(0,0)$).
    
    So in order to obtain continuity of the product map we should restrict this analysis to semigroupoids for which the multiplication map is proper. Recall that a continuous map $f\colon X\to Y$ between topological spaces is \emph{proper} if $f^{-1}(K)$ is compact whenever $K\subseteq Y$ is compact.
    
    \begin{example}
        Let $\theta$ be a continuous and associative $\land$-preaction of a finite inverse semigroupoid $S$ on a locally compact, globally Hausdorff space $X$, and assume that $\dom(\theta_s)$ is closed for every $s\in S$. Then the product map of $S\ltimes X$ is proper. 
        
        Indeed, let $p\colon S\ltimes X\to X$ be the second coordinate projection. Note that $(S\ltimes X)^{(2)}$ is a closed subset of $S\times X\times S\times X$. Let $\mu$ be the product map of $S\ltimes X$. If $K\subseteq S\ltimes X$ is compact then it is also closed, and so $\mu^{-1}(K)$ is a closed subset of the compact $\bigcup_{s,t\in S}\left\{s\right\}\times\theta_t(p(K)\cap\dom(\theta_t))\times\left\{t\right\}\times p(K)$, and hence $\mu^{-1}(K)$ is compact.
    \end{example}
    
    \begin{proposition}\label{prop:compact.open.topology.is.algebra.topology}
        Suppose that $\pi\colon\Lambda\to \Gamma$ is a continuous $R$-bundle, where $\Gamma$ is an étale, globally Hausdorff, locally $R$-normal semigroupoid with proper multiplication map. Then the compact-open topology makes $\mathcal{A}(\pi)$ into a topological algebra.
    \end{proposition}
    
    Before proving this proposition, we need a lemma which allows us to see $\Gamma^{(2)}$ locally as a product of $\Gamma$ and a finite set.
    
    \begin{lemma}\label{lem:proper.multiplication.finite.preimage}
        Let $X$ and $Y$ be Hausdorff and locally compact spaces, and $\mu\colon X\to Y$ a local homeomorphism. Then the following are equivalent:
        \begin{enumerate}[label=(\arabic*)]
            \item\label{lem:proper.multiplication.finite.preimage1} $\mu$ is a proper map.
            \item\label{lem:proper.multiplication.finite.preimage2} For every $k\in\Gamma$, the set $\mu^{-1}(k)$ is finite and the function
            \[\#\mu^{-1}\colon\Gamma\to\mathbb{N},\qquad k\mapsto\#\mu^{-1}(k)\]
            is continuous.
        \end{enumerate}
    \end{lemma}
    \begin{proof}
        \uline{\ref{lem:proper.multiplication.finite.preimage1}$\Rightarrow$\ref{lem:proper.multiplication.finite.preimage2}}: Suppose that $\mu$ is proper. Given $y\in Y$, the set $\mu^{-1}(y)$ is compact and discrete in the subspace topology of $X$ , because $\mu$ is a local homeomorphism, and thus it is finite. This proves that the function $\#\mu^{-1}$ is well-defined.
        
        To prove that $\#\mu^{-1}$ is continuous, fix $y\in Y$ and let $n=\#\mu^{-1}(y)$. Enumerate $\mu^{-1}(y)=\left\{x_1,\ldots,x_n\right\}$. Since $X$ is Hausdorff and $\mu$ is a local homeomorphism, we may take a family of pairwise disjoint neighbouhoods $U_i$ of each $x_i$ such that $\mu$ is injective on each $U_i$.
        
        Consider the neighbourhood $V\defeq\bigcap_{i=1}^n\mu(U_i)$ of $y$. Substituting each $U_i$ by $U_i\cap\mu^{-1}(V)$, we may assume that $\mu(U_i)=V$ for each $i$. Let us prove that $\#\mu^{-1}(p)\geq n$ for each $p\in V$. Indeed, for each $i$, the map $\mu$ restricts to a homeomorphism from $U_i$ to $V$, so in particular $\#(\mu^{-1}(p)\cap U_i)=1$. Since all $U_i$ are disjoint then $\#\mu^{-1}(p)\geq n$.
        
        We may now prove that $\#\mu^{-1}=n$ on a neighbourhood of $y$. If this was not the case, then the previous paragraph implies that there exists a converging net $p_i\to y$ and elements $a_i\in\mu^{-1}(p_i)\setminus\bigcup_{i=1}^nU_i$. As $Y$ is locally compact and $\mu$ is proper, we may pass to a subnet if necessary and assume that the net $(a_i)_i$ converges in $X$, say $a_i\to a$. Howeover, all $a_i$ belong to the closed set $\Gamma^{(2)}\setminus\bigcup_{i=1}^nU_i$, so $a$ does so as well. On the other hand $\mu(a)=\lim_i\mu(a_i)=\lim_i p_i=y$, so $a$ belongs to $\mu^{-1}(y)$, which is contained in $\bigcup_{i=1}^n U_i$, a contradiction.
        
        Therefore, $\#\mu^{-1}$ is locally constant, i.e., continuous.
        
        \uline{\ref{lem:proper.multiplication.finite.preimage2}$\Rightarrow$\ref{lem:proper.multiplication.finite.preimage1}}: For the converse direction, suppose $\#\mu^{-1}$ is well-defined and continuous. Let $K$ be a compact subset of $Y$. We need to prove that $\mu^{-1}(K)$ is compact. Since the function $\#\mu^{-1}$ has finite image when restricted to $K$, we may decompose $K$ into the sets $\left\{k\in K:\#\mu^{-1}(k)=n\right\}$, where $0\leq n\leq\max\left\{\#\mu^{-1}(k):k\in K\right\}$, and assume that $\#\mu^{-1}$ has a constant value $K$. Call this value $n$.
        
        Given $k\in K$, we enumerate $\mu^{-1}(k)=\left\{a_1,\ldots,a_n\right\}$, and take disjoint compact neighbourhoods $A(1,k),\ldots,A(n,k)$ of each $a_i$, respectively. As $\mu$ is a local homeomorphism, we may assume that $\mu(A(i,k))=\mu(A(j,k))$ for each $i$ and $j$. The sets $\left\{\mu(A(1,k)):k\in K\right\}$ form a neighbourhood cover of $K$, so we extract a finite subcover $\left\{\mu(A(1,k_1)),\ldots,\mu(A(1,k_N))\right\}$ of $K$. We are finished by proving that $\mu^{-1}(K)\subseteq\bigcup_{j=1}^N\bigcup_{i=1}^n A(i,k_j)$ (because the $A(i,k)$ were taken compact).
        
        Given $a\in\mu^{-1}(K)$, let $k=\mu(a)$. Choose $j$ such that $k\in\mu(A(1,k_j))$. For each $i\in\left\{1,\ldots,n\right\}$, we have $k\in\mu(A(1,k_j))=\mu(A(i,k_j))$, so $k=\mu(a_i)$ for some $a_i\in A(i,k_j)$. As the sets $A(i,k_j)$ are pairwise disjoint and $\#\mu^{-1}(k)=n$, then these elements $a_1,\ldots,a_n$ comprise all of $\mu^{-1}(a_i)$, and in particular $a=a_i\in A(i,k_j)$ for some $i$.\qedhere
    \end{proof}
    
    \begin{proof}[Proof of Proposition \ref{prop:compact.open.topology.is.algebra.topology}]
        The verification that addition and scalar multiplication of $\mathcal{A}(\pi)$ are continuous with respect to the compact-open topology is straightforward, since these operations are pointwise and all elements of $\mathcal{A}(\pi)$ are continuous (because $\Gamma$ is Hausdorff). In fact,
        if $\Gamma$ were simply a topological space (i.e., a unit groupoid), then the convolution product would also be pointwise, and the proof of its continuity would be essentially the same as that of the continuity of addition. The problem in the general case is that the convolution product is defined in terms of a sum, namely
        \[(\alpha\ast\beta)(\gamma)=\sum_{ab=\gamma}\alpha(a)\beta(b),\]
        and the number of nonzero terms in this sum depends on the argument $\gamma$. This is where Lemma \ref{lem:proper.multiplication.finite.preimage} will come into play, as it allows us to control how many nonzero terms appear in this sum.
        
        Denote by $\mu\colon\Gamma^{(2)}\to\Gamma$ the product of $\Gamma$. Then $\mu$ is proper, by hypothesis, and a local homeomorphism (Proposition \ref{prop:product.map.is.local.homeo}). 
        
        Suppose that $\alpha,\beta\in\mathcal{A}(\pi)$, and take a sub-basic neighbourhood $[K,V]$ of $\alpha\beta$, where $K\subseteq\Gamma$ is Hausdorff and $V\subseteq\Lambda$ is open. This means that for every $k\in K$ we have $\sum_{(a,b)\in\mu^{-1}(k)}\alpha(a)\beta(b)\in V$. The set $\mu^{-1}(k)$ is finite, so we may enumerate it as $\mu^{-1}(k)=\left\{(a_1,b_1),\ldots,(a_{N(k)},b_{N(k)})\right\}$. As addition and multiplication of $\Lambda$ are continuous, there exist neighbourhoods $A(k,j)$ and $B(k,j)$ of each $\alpha(a_j)$ and $\beta(b_j)$, respectively, such that $\sum_{j=1}^{N(k)}A(k,j)B(k,j)\subseteq V$. As $\alpha$ and $\beta$ are continuous, take compact neighbourhoods $P(k,j)$ and $Q(k,j)$ of $a_j$ and $b_j$ such that $\alpha(P(k,j))\subseteq A(k,j)$ and $\beta(Q(k,j))\subseteq B(k,j)$. Taking smaller neighbourhoods if necessary, we may assume, by Lemma \ref{lem:proper.multiplication.finite.preimage}, that
        \begin{enumerate}[label=(\Roman*)]
        \item\label{prop:compact.open.topology.is.algebra.topologyproofinj} $\#\mu^{-1}$ is constant on $P(k,j)Q(k,j)$ for each $j$;
        \item\label{prop:compact.open.topology.is.algebra.topologyproofdis} The sets $P(k,j)\times Q(k,j)$ are pairwise disjoint.
        \end{enumerate}
        
        The family $\left\{\bigcap_{j=1}^{N(k)}P(k,j)Q(k,j):k\in K\right\}$ is a neighbourhood cover of $K$, so take a finite subcover $\left\{\bigcap_{j=1}^{N(k_1)}P(k_1,j)Q(k_1,j),\ldots,\bigcap_{j=1}^{N(k_n)}P(k_n,j)Q(k_n,j)\right\}$. Consider the neighbourhoods
        \[A=\bigcap_{i=1}^n\bigcap_{j=1}^{N(k_i)}[P(k_i,j),A(k_i,j)]\quad\text{and}\quad B=\bigcap_{i=1}^n\bigcap_{j=1}^{N(k_i)}[Q(k_i,j),B(k_i,j)]\]
        of $\alpha$ and $\beta$, respectively. We prove that $AB\subseteq[K,V]$. Let $\alpha'\in A$ and $\beta'\in B$. Given $k\in K$, choose $i$ such that $k\in\bigcap_{j=1}^{N(k_i)}P(k_i,j)Q(k_i,j)$. This means that for each $j=1,\ldots,N(k_i)$, there is an element $(a_j,b_j)\in\left( P(k_i,j)\times Q(k_i,j)\right)\cap\Gamma^{(2)}$ such that $a_jb_j=k$. By \ref{prop:compact.open.topology.is.algebra.topologyproofdis}, these pairs $(a_j,b_j)$ are pairwise distinct, i.e., there are $N(k_i)$ of them. Since $\#\mu^{-1}$ is constant on $\bigcap_{j=1}^{N(k_i)}P(k,j)Q(k,j)$ (by \ref{prop:compact.open.topology.is.algebra.topologyproofinj}), then these elements $(a_j,b_j)$ comprise all of $\mu^{-1}(k)$. Therefore,
        \[(\alpha'\beta')(k)=\sum_{j=1}^{N(k_i)} \alpha'(a_j)\beta'(b_j)\in\sum_{j=1}^{N(k_i)}\alpha'(P(k_i,j))\beta'(Q(k_i,j))\subseteq \sum_{j=1}^{N(k_i)}A(k_i,j)B(k_i,j)\subseteq V.\]
        
        We conclude that $\alpha'\beta'\in [K,V]$, as desired.\qedhere
    \end{proof}
    
\section{Isomorphism theorems}

\subsection{Tensor products}

Tensor products are the algebraic counterpart of products of topological spaces. For example, if $X$ and $Y$ are locally compact Hausdorff spaces, then the $C^*$-algebra $C_0(X\times Y)$ is isomorphic to the $C^*$-algebraic tensor product $C_0(X)\overline{\otimes}C_0(Y)$. Recently, Rigby proved an analogous result in the setting of Steinberg algebras (see \cite[Theorem 4.3]{arxiv1811.10897}): If $R$ is a discrete unital commutative ring and $\mathcal{G}$ and $\mathcal{H}$ are ample Hausdorff groupoids, then the Steinberg algebra $R(\mathcal{G}\times\mathcal{H})$ is isomorphic to the (algebraic) tensor product $R\mathcal{G}\otimes R\mathcal{H}$. Analogous results hold in the setting of groupoid $C^*$-algebras (this is folklore; see \cite[Lemma 2.10]{arxiv1804.00967} for a proof).

Rigby's proof uses a slight generalization of the universal property of the Steinberg algebra $R\mathcal{G}$ of an ample Hausdorff groupoid $\mathcal{G}$, which states, in simple terms, that $R\mathcal{G}$ is universal for representations of the semigroup $\mathbf{KB}(\mathcal{G})$ of compact-open bisections of $\mathcal{G}$ as a Boolean inverse semigroup. See \cite[Theorem 3.10]{MR3274831} for details. The proof relies much on the Hausdorff property of $\mathcal{G}$, which allows one to take differences of compact-open bisections - this is a concrete realization of the fact that $\mathcal{G}$ is Hausdorff if and only if $\mathbf{KB}(\mathcal{G})$ is Boolean in the sense of \cite[p.140]{MR3077869}, as proven in \cite[Proposition 3.7]{MR2565546} (see also Proposition 3.23 and Theorem 3.25 of \cite{MR3077869}).

Since we do not have any analogue of such an universal property for sectional algebras we just proceed with a direct proof of our result. Although this approach requires longer and more involved computations, it has the advantage of allowing us to drop some of the ``Hausdorff'' requirements on our semigroupoids.

In the case of $C^*$-algebras (which are our model for general topological algebras), ``$C^*$-tensor products'' are defined as completions of algebraic tensor products with respect to certain norms. More specifically, if $X$ and $Y$ are compact Hausdorff topological spaces, the canonical map
\[T\colon C(X)\otimes C(Y)\to C(X\times Y),\qquad T(f\otimes g)(x,y)=f(x)g(y)\]
is not surjective, but it is injective with dense image.

Since sectional algebras do not have (in general) any appropriate topology, and much less their tensor products over arbitrary topological rings, it is not sensible to consider any kind of completion, and thus we will need to consider only purely algebraic tensor products.

Suppose that $\pi\colon\Lambda\to\Gamma$ is a continuous $R$-bundle and $\mathcal{E}$ is an étale semigroupoid. We consider the new $R$-bundle
\[\pi\times\mathcal{E}\colon\Lambda\times\mathcal{E}\to\Gamma\times\mathcal{E},\qquad(\pi\times\mathcal{E})(\lambda,e)=(\pi(\lambda),e),\]
where each fiber $(\pi\times\mathcal{E})^{-1}(\gamma,e)=\pi^{-1}(\gamma)\times\left\{e\right\}$ has the $R$-bimodule structure induced by $\pi^{-1}(\gamma)$.

We will seek an $R$-bimodule isomorphism between $\mathcal{A}(\pi\times\mathcal{E})$ and $\mathcal{A}(\pi)\tensor[_R]{\otimes}{_R}R\mathcal{E}$, where $R\mathcal{E}$ is the semigroupoid algebra of $\mathcal{E}$ as in Subsection \ref{subsectiion:semigroupoidalgebra}.

Let us finish this preliminary discussion with some considerations about the ``symmetry'' of our modules, which relates to the possible algebra structure of a tensor product (this problem was already addressed in the Introduction). Suppose that $\pi\colon\Lambda\to\Gamma$ is a continuous $R$-bundle, where $R$ is a unital topological ring. If all of the $R$-bimodules $\pi^{-1}(\gamma)$ are symmetric, then $\mathcal{A}(\pi)$ is also symmetric. In particular, if $R$ is a commutative ring and $\mathcal{E}$ is an étale semigroupoid then the semigroupoid algebra $R\mathcal{E}$ is a symmetric $R$-bimodule, since it is, by definition, the sectional algebra of the trivial (coordinate projection) bundle $R\times\mathcal{E}\to\mathcal{E}$, whose fibers have the $R$-bimodule structure coming from $R$.

The following elementary lemma will be used several times during the proof of Theorem \ref{thm:isomorphismtensorproductbundle}.

\begin{lemma}\label{lem:compositionislocalhomeoiffrightmostislocalhomeo}
    Suppose that $X,Y,Z$ are topological spaces, $f\colon X\to Y$ is continuous and $g\colon Y\to Z$ is a local homeomorphism. Then $g\circ f$ is a local homeomorphism if and only if $f$ is a local homeomorphism.
\end{lemma}

As previously mentioned, one of the main motivations for this work are the Steinberg algebras considered in \cite{arxiv1804.00396,MR3743184,MR3943326,MR3848066,MR3372123,arxiv1806.04362,MR3274831,arxiv1811.10897,MR2565546,MR3270778,MR3873946}, which are simply the (semi)groupoid algebras $R\mathcal{G}$ in the specific case where $\mathcal{G}$ is an ample groupoid and $R$ is a discrete ring. By definition, $R\mathcal{G}$ is the sectional algebra of the coordinate projection bundle $\pi\colon R\times\mathcal{G}\to\mathcal{G}$. Note that $R$ is discrete if and only if $\pi$ is a local homeomorphism.

So we should regard the étale bundles as the ``discrete ones'' ones, which will be of special interest in further results as well. Moreover, we may specify:

\begin{proposition}\label{prop:etalebundle}
    Let $\pi\colon\Lambda\to\Gamma$ be a continuous rigid semigroupoid homomorphism, where $\Gamma$ is an étale semigroupoid. Consider the graph structure on $\Lambda$ given by $\so_\Lambda=\so_{\Gamma}\circ\pi$ and $\ra_\Lambda=\ra_\Gamma\circ\pi$, which is compatible with the semigroupoid structure of $\Lambda$ because $\pi$ is rigid. Then the following are equivalent:
    \begin{enumerate}[label=(\arabic*)]
        \item $\Lambda$ is étale (with the graph structure $(\so_\Lambda,\ra_\Lambda)$);
        \item $\pi$ is a local homeomorphism.
    \end{enumerate}
\end{proposition}
\begin{proof}
    By Lemma \ref{lem:compositionislocalhomeoiffrightmostislocalhomeo}, the source map of $\Lambda$, $\so_\Lambda=\so_\Gamma\circ\pi$, is a local homeomorphism if and only if $\pi$ is a local homeomorphism, and similarly for $\ra_\Lambda$. This proves the equivalence (1)$\iff$(2).\qedhere
\end{proof}

Note that if $R$ is a (possibly non-commutative) unital discrete ring then, as long as $0\neq 1$ in $R$, the locally $R$-unital étale semigroupoids are precisely the ample (i.e., zero-dimensional) ones. Our first main theorem follows below.

\begin{theorem}\label{thm:isomorphismtensorproductbundle}
    Let $R$ be a unital topological ring and $\pi\colon\Lambda\to\Gamma$ be a continuous $R$-bundle. Assume that $\Gamma$ and $\mathcal{E}$ are locally $R$-normal étale semigroupoids. Then there exists a unique $R$-bimodule homorphism $T\colon\mathcal{A}(\pi)\tensor[_R]{\otimes}{_R} R\mathcal{E}\to\mathcal{A}(\pi\times\mathcal{E})$ given on pure tensors by
    \[T(\alpha\otimes f)(\gamma,e)=(\alpha(\gamma)f(e),e)\ntag\label{eq:thm:isomorphismtensorproductbundle}\]
    for all $\alpha\in\mathcal{A}(\pi)$, $R\mathcal{E}$ and $(\gamma,\pi)\in\Gamma\times\mathcal{E}$.
    \begin{enumerate}[label=(\alph*)]
        \item\label{thm:isomorphismtensorproductbundlesym} If $R$ is commutative and $\pi^{-1}(\gamma)$ is symmetric for all $\gamma\in\Gamma$, then $T$ is an $R$-algebra homomorphism.
        \item If $\Gamma$ and $\mathcal{E}$ are ample and $\pi$ is a local homeomorphism (i.e., $\Lambda$ is étale as in Proposition \ref{prop:etalebundle}), then $T$ is surjective.
        \item If $R$ is discrete and $\mathcal{E}$ is ample and Hausdorff, then $T$ is injective.
        \item If $R$ is a field, then $T$ is injective.
    \end{enumerate}
\end{theorem}

\begin{remark}
\begin{enumerate}[label=(\arabic*)]
    \item The isomorphism $T$ of the theorem above actually preserves several other module structures of $\mathcal{A}(\pi\times\mathcal{E})$ and $\mathcal{A}(\pi)\tensor[_R]{\otimes}{_R} R\mathcal{E}$. For example, $\mathcal{A}(\pi\times\mathcal{E})$ has a canonical left $\mathcal{A}(\pi)$-module structure as follows: Each section $\overline{F}\in\mathcal{A}(\pi\times\mathcal{E})$ is of the form $\overline{F}(\gamma,e)=(F(\gamma,e),e)$ for some unique function $F\colon\Gamma\times E\to\Lambda$. Given $\alpha\in\mathcal{A}(\pi)$, set
    \[(\alpha\ast \overline{F})(\gamma,e)=\sum_{\gamma_1\gamma_2=\gamma}(\alpha(\gamma_1)F(\gamma_2,e),e).\]
    On the other hand, $\mathcal{A}(\pi)\tensor[_R]{\otimes}{_R} R\mathcal{E}$ has a canonical left $\mathcal{A}(\pi)$-module structure as well. Similarly, both $\mathcal{A}(\pi\times\mathcal{E})$ and $\mathcal{A}(\pi)\tensor[_R]{\otimes}{_R}R\mathcal{E}$ have right $R\mathcal{E}$-module structures, and $T$ is a $(\mathcal{A}(\pi),R\mathcal{E})$-bimodule homomorphism as well.
    
    \item Suppose that $R$ is a commutative ring, $A=\oplus_{g\in G}A_g$ is a $G$-graded symmetric $R$-algebra and $B=\oplus_{h\in H}B_h$ is an $H$-graded symmetric $R$-algebra, where $G$ and $H$ are semigroupoids.
    
    Then the tensor product $A\tensor[_R]{\otimes}{_R} B$, regarded as a symmetric $R$-algebra, is $(G\times H)$-graded, with homogeneous components $(A\tensor[_R]{\otimes}{_R} B)_{(g,h)}=A_g\tensor[_R]{\otimes}{_R} B_h$.
    
    In particular, under the same conditions as in item \ref{thm:isomorphismtensorproductbundlesym}, suppose that $c\colon\Gamma\to G$ and $d\colon\mathcal{E}\to H$ are semigroupoid gradings. Then $\mathcal{A}(\pi)$ is $G$-graded and $R\mathcal{E}$ is $H$-graded, so the tensor product $\mathcal{A}(\pi)\tensor[_R]{\otimes}{_R}R\mathcal{E}$ is $(G\times H)$-graded as above. On the other hand, we have a semigroupoid grading $c\times d\colon \Gamma\times\mathcal{E}\to G\times H$, which induces a $(G\times H)$-grading on $\mathcal{A}(\pi\times\mathcal{E})$.
    
    In this case, the homomorphism $T$ of Theorem \ref{thm:isomorphismtensorproductbundle} above is a $(G\times H)$-graded homomorphism.
    
    \item If $R$ is discrete (possibly non-commutative), $\pi$ is a local homeomorphism and $\Gamma$ and $\mathcal{E}$ are ample, with $\mathcal{E}$ Hausdorff, then $T$ as above is an isomorphism, and thus induces an $R$-algebra structure on the tensor product $\mathcal{A}(\pi)\tensor[_R]{\otimes}{_R}R\mathcal{E}$. However the product is not given by ``entrywise product of pure tensors'' (i.e., in general we may have $(\alpha\otimes f)(\beta\otimes g)\neq(\alpha\ast\beta)\otimes( f\ast g)$).
    \item Similarly, an obvious variation of the theorem above yields an $R$-bimodule homomorphism $ R\mathcal{E}\tensor[_R]{\otimes}{_R}\mathcal{A}(\pi)\to \mathcal{A}(\pi\times\mathcal{E})$. If $R$ is discrete, $\pi$ is a local homeomorphism, $\Gamma$ and $\mathcal{E}$ are ample and $\mathcal{E}$ is Hausdorff, then we obtain $R$-bimodule isomorphisms $\mathcal{A}(\pi)\tensor[_R]{\otimes}{_R}R\mathcal{E}\cong\mathcal{A}(\pi\times\mathcal{E})\cong R\mathcal{E}\tensor[_R]\otimes{_R}\mathcal{A}(\pi)$. However the isomorphism between the tensor products is not simply ``inversion of pure tensors'' (i.e., it is not given by $\alpha\otimes f\mapsto f\otimes \alpha$)
    \end{enumerate}
\end{remark}

\begin{proof}[{Proof of Theorem \ref{thm:isomorphismtensorproductbundle}}]
    Although the existence and uniqueness of the map $T$ as in Equation \eqref{eq:thm:isomorphismtensorproductbundle} follows immediately from the universal property of tensor products with respect to $R$-balanced maps, it is still necessary to check that for each $(\alpha,f)\in \mathcal{A}(\pi)\times R\mathcal{E}$, the map $T(\alpha\otimes f)\colon\Gamma\times\mathcal{E}\to \Lambda\times\mathcal{E}$ is an element of $\mathcal{A}(\pi\times\mathcal{E})$.
    
    It is clear that $T(\alpha\otimes f)$ is a section of $\pi\times\mathcal{E}$, so the main issue is to prove that it is a combination of sections of $\pi\times\mathcal{E}$ which are ``continuously and compactly supported on Hausdorff subsets of $\Gamma\times\mathcal{E}$''.
    
    If $V\subseteq\Gamma$ and $W\subseteq\mathcal{E}$ are open and Hausdorff in $\Gamma$ and $\mathcal{E}$, respectively, $\alpha\in C_c(V,\pi)$ and $f\in C_c(W,R)$, then we have $T(\alpha\otimes f)\in C_c(V\times W,\pi\times\mathcal{E})$, which belongs to $\mathcal{A}(\pi\times\mathcal{E})$. As $\mathcal{A}(\pi)$ is generated by the union of all sets $C_c(V,\pi)$ for $V$ Hausdorff, and similarly for $R\mathcal{E}$, we may take arbitrary linear combinations and conclude that $T(\alpha\otimes f)\in \mathcal{A}(\pi\times\mathcal{E})$ for arbitrary $\alpha$ and $f$.
    
    \begin{enumerate}[label=(\alph*)]
        \item If we assume that $R$ is commutative and $\pi^{-1}(\gamma)$ is a symmetric $R$-bimodule for each $\gamma$, it readily follows from all relevant definitions that $T(\alpha\otimes f)T(\beta\otimes g)=T(\alpha\beta\otimes fg)$ for all $\alpha,\beta\in\mathcal{A}(\pi)$ and $f,g\in\mathcal{A}_R(\mathcal{E})$, so $T$ is an algebra homomorphism.
        \item We assume that $\Gamma$ and $\mathcal{E}$ are ample and that $\pi$ is a local homeomorphism. By Lemma \ref{lem:sectionalalgebraisgeneratedbybasis}, it is enough to prove that any $\overline{F}\in C_c(V\times W,\pi\times\mathcal{E})$ where $V$ and $W$ are open bisections of $\Gamma$ and $\mathcal{E}$, respectively, belongs to the image of $T$.
        
        Consider the restriction $\overline{F}|_{V\times W}$, which is a continuous map. As $\overline{F}$ is a section of $\pi\otimes\mathcal{E}$, then $(\pi\otimes\mathcal{E})\circ \overline{F}|_{V\times W}=\id_{V\times W}$, so by Lemma \ref{lem:compositionislocalhomeoiffrightmostislocalhomeo}, $\overline{F}|_{V\times W}$ is a local homeomorphism, and in particular it is an open map. Consider the compact $K\defeq\supp(\overline{F})\cap (V\times W)$, and let $F$ be the composition of $\overline{F}|_V$ with the projection $\Lambda\times\mathcal{E}\to\Lambda$. Then $F$ is a continuous, open map from $V\times W$ to $\Lambda$, with image the open set $F(V\times W)$.
    
        The set $F(K)$ is compact in $\Lambda$, and so it may be covered by finitely many open subsets $A_1,\ldots,A_n$ of $\Lambda$ on which $\pi$ is injective. Since $F$ is continuous, then the sets $F^{-1}(A_i)$ form a cover of $K$. We then consider a finer finite cover by ``boxes'' of the form $V_j\times W_j$, $j=1,\ldots,m$, where $V_j\subseteq V$ and $W_j\subseteq W$ are compact-open (here is where we use that $\Gamma$ and $\mathcal{E}$ are ample). In fact, as $V$ and $W$ are Hausdorff then these sets $V_j$ and $W_j$ sets are clopen in $V$ and $W$, respectively. Taking appropriate intersections and differences (which preserve clopen sets) of the boxes $V_j\times W_j$, and rewriting them as disjoint unions of smaller boxes, we may moreover assume that these boxes $V_j\times W_j$ are pairwise disjoint (this is the same procedure as when one proves that boxes form a semiring of subsets of $V\times W$, so we ommit the details). Of course, we may also assume that $W_j\neq\varnothing$ for each $j$.
    
        We now prove that for each $j$, the value of $F$ inside $V_j\times W_j$ depends only on the first entry, i.e., that if $(\gamma,e_0),(\gamma,e_1)\in V_j\times W_j$, then $F(\gamma,e_0)=F(\gamma,e_1)$. Indeed, first choose $i$ such that $V_j\times W_j\subseteq \pi^{-1}(A_i)$. Then $\pi(F(\gamma,e_0))=\gamma=\pi(F(\gamma,e_1))$, because $\overline{F}$ is a section of $\pi\times\mathcal{E}$. Since $\pi$ is injective on $A_i$, which contains $F(V_j\times W_j)$, then $F(\gamma,e_0)=F(\gamma,e_1)$.
    
        Thus we define, for each $j$, the maps $f_j\colon V_j\to\Lambda$ as $f_j(\gamma)=F(\gamma,e)$, where $e$ is an arbitrary element of $W_j$. Then $f_j$ is continuous on the compact-open set $V_j$. We extend $f_j$ as zero on $\Gamma\setminus V_j$, and so $f_j\in C_c(V_j,\pi)\subseteq \mathcal{A}(\pi)$. Similarly, as $W_j$ is also compact-open, then the characteristic function $1_{W_j}$ of $W_j$, from $\mathcal{E}$ to $R$, belongs to $R\mathcal{E}$.
    
        Now let us prove that $\overline{F}=\sum_j T(f_j\otimes 1_{W_j})$, or equivalently that
        \[F(\gamma,e)=\sum_jf_j(\gamma)1_{W_j}(e)\qquad\text{for all }(\gamma,e)\in\Gamma\times\mathcal{E}.\ntag\label{eq:proofofthm:isomorphismtensorproductbundle}\]
        There are two cases to consider:
        \begin{itemize}
            \item If $(\gamma,e)$ does not belong to any of the sets $V_j\times W_j$, then in particular it does not belong to $K$, so both sides of Equation \eqref{eq:proofofthm:isomorphismtensorproductbundle} are zero.
        \item If $(\gamma,e)$ belongs to $V_j\times W_j$ for some $j$, then in fact such $j$ is unique since the sets $V_j\times W_j$ are pairwise disjoint. In this case the equality of Equation \eqref{eq:proofofthm:isomorphismtensorproductbundle} follows by definition of $f_j$.
        \end{itemize}
        Therefore, $T$ is surjective.
    \item We now assume that $R$ is discrete and $\mathcal{E}$ is ample and Hausdorff, and prove that $T$ is injective.
    
        As $R$ is discrete and $\mathcal{E}$ is ample and Hausdorff, then $R\mathcal{E}$ is generated as a left $R$-module by functions $1_W$, where $W\subseteq\mathcal{E}$ is a compact-open subset of $\mathcal{E}$. Thus every element of $\mathcal{A}(\pi)\otimes\mathcal{A}_R(\mathcal{E})$ is a sum of the form $\sum_{i=1}^n \alpha_i\otimes 1_{W_i}$, where $\alpha_i\in\mathcal{A}_R(\pi)$ and the $W_i$ are compact-open subsets of $\mathcal{E}$. In fact, we may also assume that such sets $W_i$ are pairwise disjoint, by taking appropriate intersections and set differences among them, similarly to how we did in the previous item.
        
        More formally, as the sets $W_i$ are clopen, then there exists a finite refinement $\mathscr{B}=\left\{B_j:j\right\}$ of $\left\{W_1,\ldots,W_n\right\}$ by nonempty, pairwise disjoint clopen subsets. This refinement $\mathscr{B}$ will have the following properties:
        \begin{itemize}
            \item For any $i$ and $j$, $B_j\cap W_i\neq\varnothing$ if and only if $B_j\subseteq W_i$;
            \item More generally, for any $i$ and $j$ and any $e\in B_j$, we have $e\in W_i$ if and only if $B_j\subseteq W_i$;
            \item For any $i$, $W_i$ is the disjoint union of all $B_j$ contained in $W_i$. In symbols, $W_i=\sqcup_{j:B_j\subseteq W_i}B_j$.
        \end{itemize}
        For example, given a subset $P$ of $\left\{1,\ldots,n\right\}$, let $B_P=\bigcap_{i\in P}W_i\setminus\bigcup_{i\not\in P}W_i$. The family $\mathscr{B}=\left\{B_P:P\in 2^{\left\{1,\ldots,n\right\}}\right\}\setminus\left\{\varnothing\right\}$ has at most $2^n$ elements and the desired properties.
        
        In any case, we can rewrite $1_{W_i}=\sum_{B\in\mathscr{B}:B\subseteq W_i}1_B$, and so
        \[\sum_{i=1}^n\alpha_i\otimes 1_{W_i}=\sum_{B\in\mathscr{B}}\left(\sum_{i:B\subseteq W_i}\alpha_i\right)\otimes 1_B.\]

        We may now prove that $T$ is injective. Suppose that $x\in\ker T$. The argument above shows that we can write $x$ as a sum $x=\sum_{i=1}^n\alpha_i\otimes1_{W_i}$, where the sets $W_1,\ldots,W_n$ are nonempty and pairwise disjoint. Let $j$ be fixed and choose an arbitrary $e\in W_j$. Since the sets $W_1,\ldots,W_n$ are pairwise disjoint, then for all $\gamma\in\Gamma$,
        \[(0,e)=T(x)(\gamma,e)=(\sum_{i=1}^n\alpha_i(\gamma)1_{W_i}(e),e)=(\alpha_j(\gamma),e),\]
        hence $\alpha_j=0$ for each $j$, and thus $x=\sum_{i=1}^n 0\otimes 1_{W_i}=0$. Therefore $T$ is injective.
    \item We now assume that $R$ is a field, and prove that $T$ is injective in this case.
    
    Let $x\in\ker T$. Then $x$ may be written as $x=\sum_{i=1}^n\alpha_i\otimes f_i$, where the elements $\alpha_i$ are linearly independent with respect to the right $R$-vector space structure of $\mathcal{A}(\pi)$ (e.g. elements of a prescribed basis). Then for all $e\in\mathcal{E}$ and all $\gamma\in\Gamma$ we have
    \[(0,e)=T(x)(\gamma,e)=(\sum_{i=1}^n\alpha_i(\gamma)f_i(e),e)\]
    so $0=\sum_{i=1}^n\alpha_i f_i(e)$ in $\mathcal{A}(\pi)$, for each $e\in\mathcal{E}$. As the $\alpha_i$ are linearly independent then $f_1(e)=\cdots=f_n(e)=0$, for each $e\in\mathcal{E}$. Thus $f_1=\cdots=f_n=0$, so $x=0$. Therefore $T$ is injective.\qedhere
    \end{enumerate}
\end{proof}

The theorem above may be seen as a simultaneous generalization to both Proposition 4.1 and Theorem 4.3 of \cite{arxiv1811.10897}.

\begin{corollary}
Let $R$ and be a discrete commutative ring and $A$ an $R$-algebra. Then given any ample Hausdorff semigroupoid $\Gamma$, the semigroupoid algebra $A\Gamma$ is isomorphic as an $R$-bimodule to $A\tensor[_R]{\otimes}{_R}(R\Gamma)$. If $R$ is commutative and $A$ is a symmetric $R$-algebra then this isomorphism is both an $R$-algebra and an $A$-algebra isomorphism.
\end{corollary}

\begin{proof}
By definition, $A\Gamma$ is the sectional algebra of the bundle $A\times\Gamma\to\Gamma$ given by the second coordinate projection. In fact, we may consider the ``trivial bundle'' $\pi^A\colon A\to\left\{1\right\}$ to the one-element group, which satisfies $\mathcal{A}(\pi^A)\cong A$. The coordinate projection bundle $A\times\Gamma\to\Gamma\cong\left\{1\right\}\times\Gamma$ may be obviously identified with $\pi^A\times\Gamma$. Theorem \ref{thm:isomorphismtensorproductbundle} yields an explicit isomorphism of $R$-bimodules/$R$-algebra/$A$-algebra in each relevant case
\[A\Gamma\cong\mathcal{A}(\pi^A\times\Gamma)\cong\mathcal{A}(\pi^A)\tensor[_R]{\otimes}{_R}(R\Gamma)\cong A\tensor[_R]{\otimes}{_R}(R\Gamma).\qedhere\]
\end{proof}

\begin{corollary}\label{cor:isomorphismtensorproduct}
    Let $\Gamma_1$ and $\Gamma_2$ be ample semigroupoids, $R$ be a discrete unital ring, and $A$ be a discrete $R$-algebra. Suppose that at least one of the following conditions holds:
    \begin{enumerate}[label=(\Roman*)]
        \item\label{cor:isomorphismtensorproductHaus} $\Gamma_2$ is Hausdorff.
        \item\label{cor:isomorphismtensorproductfield} $R$ is a field.
    \end{enumerate}
    Then $A(\Gamma_1\times\Gamma_2)$ is isomorphic as an $R$-bimodule to $(A\Gamma_1)\tensor[_R]{\otimes}{_R}(R\Gamma_2)$. If $R$ is commutative, these symmetric algebras are isomorphic.
\end{corollary}

Even in the case of groupoids and commutative rings, this already yields a generalization of \cite[Theorem 4.3]{arxiv1811.10897}, since only $\Gamma_2$ is required to be Hausdorff.

\begin{proof}
    By definition, $A\Gamma_1$ is the sectional algebra of the second-coordinate projection bundle $\pi_1\colon A\times\Gamma_1\to\Gamma_1$, and similarly $A(\Gamma_1\times\Gamma_2)$ is the sectional algebra of the second and third coordinate projection bundle $A\times \Gamma_1\times\Gamma_2\to\Gamma_1\times\Gamma_2$, which actually is the same as $\pi_1\times\Gamma_2$. Under either of the Conditions \ref{cor:isomorphismtensorproductHaus} or \ref{cor:isomorphismtensorproductfield}, Theorem \ref{thm:isomorphismtensorproductbundle} yields explicit isomorphisms
    \[A(\Gamma_1\times\Gamma_2)=\mathcal{A}(\pi_1\times\Gamma_2)\cong\mathcal{A}(\pi_1)\tensor[_R]{\otimes}{_R} (R\Gamma_2)=(A\Gamma_1)\tensor[_R]{\otimes}{_R}(R\Gamma_2).\qedhere\]
\end{proof}

\subsection{Sectional algebras of semidirect product bundles as naïve crossed products}

\begin{definition}\label{def:landpreactiononRbundle}
    A $\land$-preaction $\theta$ of an étale inverse semigroupoid $\mathcal{S}$ on a continuous $R$-bundle $\pi\colon\Lambda\to\Gamma$ consists of two $\land$-preactions $\theta^\Gamma$ and $\theta^\Lambda$ of $\mathcal{S}$ on $\Gamma$ and $\Lambda$, respectively, satisfying:
    \begin{enumerate}[label=(\roman*)]
    \item\label{def:landpreactiononRbundle.1} for all $s\in\mathcal{S}$, $\pi\circ\theta_s^\Lambda=\theta_s^\Gamma\circ\pi$;
    \item\label{def:landpreactiononRbundle.2} $\theta^\Lambda$ preserves all of the relevant $R$-bimodule structure, in the sense that if $s\in\mathcal{S}$ and $a\in\dom(\theta^\Gamma_s)$, then $\theta^\Lambda_s$ restricts to an $R$-bimodule isomorphism from $\pi^{-1}(a)$ onto $\pi^{-1}(\theta^\Gamma_s(a))$.
    \end{enumerate}
    This in particular means that:
    \begin{enumerate}[label=(\roman*)]\setcounter{enumi}{2}
        \item $\dom(\theta^\Lambda_s)=\pi^{-1}(\dom(\theta^\Gamma_s))$;
        \item $\pi(\theta^\Lambda_s(a))=\theta^\Gamma_s(\pi(a))$, in the sense that either side is defined if and only if the other one is defined, in which case they coincide;
        \item $\dom(\theta^\Lambda_s)+\dom(\theta^\Lambda_s)\subseteq\dom(\theta^\Lambda_s)$.
    \end{enumerate}
    We say that $\theta$ is continuous if both $\theta^\Gamma$ and $\theta^\Lambda$ are continuous, and similarly for open/associative.
\end{definition}

\begin{remark}
    The statement in \ref{def:landpreactiononRbundle.2} is sensible because of \ref{def:landpreactiononRbundle.1}. More precisely, \ref{def:landpreactiononRbundle.1} alone already implies that for all $s\in\mathcal{S}$ and $a\in\dom(\theta^\Gamma_s)$ we have $\theta^\Lambda_s(\pi^{-1}(a))=\pi^{-1}(\theta^\Gamma_s(a))$, so that $\theta^\Lambda_s$ already restricts to a bijection from $\pi^{-1}(a)$ onto $\pi^{-1}(\theta^\Gamma_s(a))$, which are $R$-bimodules. Thus it makes sense to require this bijection to be an $R$-bimodule homomorphism.
\end{remark}

We may omit superscripts and write simply $\theta$ for either $\theta^\Lambda$ or $\theta^\Gamma$, whenever there is no risk of confusion.

\begin{itemize}
    \item \uline{\textbf{Convention}}: From now on and until the end of this subsection, we fix a continuous, open and associative $\land$-preaction $\theta$ of an étale, locally $R$-normal inverse semigroupoid $\mathcal{S}$ on a continuous $R$-bundle $\pi\colon\Lambda\to\Gamma$, where $\Gamma$ is étale and locally $R$-normal, and moreover \uline{we assume that $\mathcal{S}\ltimes\Gamma$ is open as a subset of $\mathcal{S}\times\Gamma$}.
\end{itemize}

With this, we may perform two procedures:
\begin{itemize}
    \item First construct the semidirect products $\mathcal{S}\ltimes\Lambda$ and $\mathcal{S}\ltimes\Gamma$, induce a new $R$-bundle $\mathcal{S}\ltimes\pi\colon\mathcal{S}\ltimes\Lambda\to\mathcal{S}\ltimes\Gamma$ and take its sectional algebra $\mathcal{A}(\mathcal{S}\ltimes\pi)$;
    \item First construct the sectional algebra $\mathcal{A}(\pi)$, induce a $\land$-preaction of $\mathcal{S}$ on $\mathcal{A}(\pi)$ and then consider the naïve crossed product $\mathcal{A}(\pi)\rtimes\mathcal{S}$.
\end{itemize}
The current goal is to prove that, under appropriate technical conditions, these two procedures yield isomorphic algebras. In other words, ``sectional algebras'' intertwine ``semidirect products'' and ``naïve crossed products''.

\begin{denv*}{Semidirect product bundles}
Consider the semidirect products $\mathcal{S}\ltimes\Lambda$ and $\mathcal{S}\ltimes\Gamma$. Since $\mathcal{S}\ltimes\Gamma$ is an open subset of $\mathcal{S}\times\Gamma$ then $\theta^\Gamma$ is an open $\land$-preaction and $\mathcal{S}\ltimes\Gamma$ is an étale semigroupoid. (See the discussion succeeding \cite[Proposition 3.10]{arxiv1902.09375}.)

We thus define the new $R$-bundle
\[\mathcal{S}\ltimes\pi\colon\mathcal{S}\ltimes\Lambda\to\mathcal{S}\ltimes\Gamma,\qquad (\mathcal{S}\ltimes\pi)(s,x)=(s,\pi(x)).\] 
where, for all $(s,\gamma)\in\mathcal{S}\ltimes\Gamma$, the preimage
\[(\mathcal{S}\ltimes\pi)^{-1}(s,\gamma)=\left\{s\right\}\times\pi^{-1}(\gamma)\]
carries the $R$-bimodule structure induced by that of $\pi^{-1}(\gamma)$.
\end{denv*}

\begin{denv*}{The induced $\land$-preaction on $\mathcal{A}(\pi)$}

As we are assuming that $\mathcal{S}\ltimes\Gamma$ is open in $\mathcal{S}\times X$, then for every $s\in\mathcal{S}$ the set
\[\dom(\theta_s)=\left\{a\in\Gamma:(s,a)\in\mathcal{S}\ltimes\Gamma\right\},\]
is open in $\Gamma$.

Consider the sectional algebra $\mathcal{A}(\pi)$. We define a $\land$-preaction $\Theta$ of $\mathcal{S}$ on $\mathcal{A}(\pi)$ by setting, for all $s\in\mathcal{S}$,
\begin{enumerate}[label=(\roman*)]
    \item $\dom(\Theta_s)=\bigcup\left\{C_c(V,\pi):V\text{ Hausforff, }V\subseteq\dom(\theta^{\Gamma}_s)\right\}$;
    \item \[\Theta_s(f)(\gamma)=\begin{cases}
    \theta^{\Lambda}_s(f(\theta^\Gamma_{s^*}(\gamma))),
        &\text{if }\gamma\in\dom(\theta^\Gamma_{s^*}),\\
    0_\gamma,
        &\text{otherwise,}
    \end{cases}\]
    whenever $f\in\dom(\Theta_s)$.
\end{enumerate}

\begin{lemma}
    If $\theta$ is associative then $\Theta$ is associative.
\end{lemma}
\begin{proof}
    To determine the associativity of $\Theta$ we need to verify that \[\left[\Theta_t^*(f\Theta_t(g))h\right](a)=\left[\Theta_t^*(f\Theta_t(gh))\right](a)\ntag\label{eq:conditionforassociativityofbigtheta}\]
    for all $a\in\Gamma$, whenever $\supp(f)\subseteq\dom(\theta_s)$, $\supp(g)\subseteq\dom(\theta_t)$ and $\supp(h)\subseteq\ran(\theta_u)$, where $stu$ is defined in $\mathcal{S}$.
    
    If $a\not\in\dom(\theta_t)$ then both sides of Equation \eqref{eq:conditionforassociativityofbigtheta} are equal to $0_t$, so we assume $a\in\dom(\theta_t)$. On one hand, a straightforward usage of the definition of the product of $\mathcal{A}(\pi)$ yields
    \begin{align*}
        \left[\Theta_t^*(f\Theta_t(g))h\right](a)\hspace{-20pt}&\\
        &=\sum\left\{\theta_t^*(f(d)\theta_t(g(\theta_t^*(e)))h(c):bc=a,b\in\dom(\theta_t),de=\theta_t(b),d\in\dom(\theta_s),e\in\ran(\theta_t)\right\}\\
        &=\sum\left\{\theta_t^*(f(d)\theta_t(g(\theta_t^*(e)))h(c):\theta_{t^*}(de)c=a,d\in\dom(\theta_s),e\in\ran(\theta_t)\right\},\ntag\label{eq:conditionforassociativityofbigthetaleft},
    \end{align*}
    where the last equality follows from $de=\theta_t(b)$.
    On the other hand, a similar computation gives us
    \begin{align*}
        \Theta_t^*(f\Theta_t(gh))(a)
        &=\sum\left\{\theta_t^*(f(x)\theta_t(g(z)h(w))):xy=\theta_t(a),x\in\dom(\theta_s),zw=\theta_{t^*}(y),z\in\dom(\theta_t)\right\}\\
        &=\sum\left\{\theta_t^*(f(d)\theta_t(g(\theta_{t^*}(e))h(c))):d\theta_t(\theta_{t^*}(e)c)=\theta_t(a),d\in\dom(\theta_s),e\in\ran(\theta_t)\right\}\ntag\label{eq:conditionforassociativityofbigthetaright},
    \end{align*}
    where the last equality follows from the substitutions $d=x$, $w=c$, $y=\theta_t(dw)$ and $z=\theta_{t^*}(e)$. As $\theta^\Lambda$ is associative, then the respective terms of each of the sums in \eqref{eq:conditionforassociativityofbigthetaleft} and \eqref{eq:conditionforassociativityofbigthetaright} are equal. As $\theta^\Gamma$ is associative, the conditions ``$d\theta_t(\theta_{t^*}(e)c)=\theta_t(a)$'' and ``$\theta_{t^*}(de)c=a$'' are equivalent. Therefore the elements in \eqref{eq:conditionforassociativityofbigthetaleft} and \eqref{eq:conditionforassociativityofbigthetaright} are equal, so $\Theta$ is associative.\qedhere
\end{proof}

\end{denv*}

If $\Gamma$ is Hausdorff and has a proper multiplication, and we consider the sectional algebra $\mathcal{A}(\pi)$ with the topology of compact-open convergence, then the $\land$-preaction $\Theta$ of $\mathcal{S}$ on $\mathcal{A}(\pi)$ is continuous.

We are now ready to state our second main theorem.

\begin{theorem}\label{thm:isomorphism.of.naive.crossed.product}
    Suppose that $\theta=(\theta^\Gamma,\theta^\Lambda)$ is a continuous, associative $\land$-preaction of an inverse semigroupoid $\mathcal{S}$ on a continuous $R$-bundle $\pi\colon\Lambda\to\Gamma$, where both $\mathcal{S}$ and $\Gamma$ are étale and locally $R$-normal, and that $\mathcal{S}\ltimes\Gamma$ is open in $\mathcal{S}\times\Gamma$. Suppose, moreover, that one of the following conditions holds:
    \begin{enumerate}[label=(\Roman*)]
        \item $\mathcal{S}$ is discrete; In this case we regard the sectional algebra $\mathcal{A}(\pi)$ as a discrete algebra.
        \item $\Gamma$ and $\mathcal{S}$ are Hausdorff, and the product map of $\Gamma$ is proper; In this case we consider $\mathcal{A}(\pi)$ as a topological algebra with the compact-open topology.
    \end{enumerate}
    
    Then the sectional algebra $\mathcal{A}(\mathcal{S}\ltimes\pi)$ is isomorphic to the naïve crossed product $\mathcal{S}\bigstar\mathcal{A}(\pi)$.
\end{theorem}
\begin{proof}
Let us describe the main idea of the proof. An element of $\mathcal{A}(\id_\mathcal{S}\times\pi)$ is a function from $\mathcal{S}\ltimes\Gamma$ to $\mathcal{S}\ltimes\Lambda$ which preserves the first coordinate, so it may be regarded simply as a function $f$ from $\mathcal{S}\ltimes\Gamma$ to $\Lambda$ which satisfies $\pi(f(s,t))=t$. On the other hand, by Definition \ref{def:alternative.definition.for.naïve.crossed.product}, an element from $\mathcal{S}\bigstar\mathcal{A}(\pi)$ is a function from $\mathcal{S}$ to $\mathcal{A}(\pi)$, i.e., to a set of functions from $\Gamma$ to $\Lambda$.

In other words, we see $\mathcal{A}(\id_{\mathcal{S}}\times\pi)$ as a subset of the function space $\Lambda^{\mathcal{S}\ltimes\Gamma}$, and $\mathcal{S}\bigstar\mathcal{A}(\pi)$ as a subset of the function space $\left(\Lambda^\Gamma\right)^\mathcal{S}$. Thus the desired isomorphism is just a translation of the fact that, for sets $X,Y,Z$, $X^{Y\times Z}$ and $\left(X^Y\right)^Z$ are in natural bijection.

We define $\Phi\colon\mathcal{A}(\mathcal{S}\ltimes\pi)\to\mathcal{S}\bigstar\mathcal{A}(\pi)$ as follows: Given a section $f=f^1\times f^2\colon\mathcal{S}\ltimes\Gamma\to\mathcal{S}\ltimes\Lambda$ of $\mathcal{S}\times\pi$, define $\Phi(f)\colon\mathcal{S}\to\mathcal{A}(\pi)$ as
\[\Phi(f)(s)(\gamma)=\begin{cases}f^2(s,\gamma),&\text{if }\gamma\in\dom(\theta_s),\\
0_\gamma,&\text{otherwise.}
\end{cases}.\]
A priori, it is not immediate that $\Phi$ is well-defined, as we first need to guarantee that $\Phi(f)(s)\in\mathcal{A}(\pi)$ for all $s\in\mathcal{S}$, and then that $\Phi(f)\in\mathcal{S}\bigstar\mathcal{A}(\pi)$. This is where the additional hypotheses come into play.

The verification that $\Phi(f)(s)\in\mathcal{A}(\pi)$ for all $s\in\mathcal{S}$ can be done without any additional hypotheses. In fact, it is enough to assume that $f$ is a generating element of $\mathcal{A}(\mathcal{S}\ltimes\pi)$ (taking

A basic open set of $\mathcal{S}\ltimes\Gamma$ has the form $U\times V$, where $U\in\mathbf{B}(\mathcal{S})$ and $V\in\mathbf{B}(\Gamma)$ are open bisections (in particular, Hausdorff). Suppose that $f\in C_c(U\times V,\mathcal{S}\ltimes\pi)$. For all $s\in\mathcal{S}$ we have
\[\Phi(f)(s)(\gamma)=\begin{cases}f(s,\gamma),&\text{if }\gamma\in V,\\0_\gamma,\text{otherwise},\end{cases}\]
so $\Phi(f)(s)$ is continuous and compactly supported on $V$, i.e., $\Phi(f)(s)\in C_c(V,\pi)\subseteq\mathcal{A}(\pi)$.

\begin{enumerate}[label=(\Roman*)]
    \item If $\mathcal{S}$ is discrete and $\mathcal{A}(\pi)$ is regarded as a discrete algebra, continuity of $\Phi(f)\colon\mathcal{S}\to\mathcal{A}(\pi)$ is trivial.
    \item The second case is more interesting, where $\mathcal{S}$ is not necessarily discrete and $\mathcal{A}(\pi)$ is endowed with the compact-open topology. In this case we slightly improve the definition of $\Phi$: If $f\in\mathcal{A}(\mathcal{S}\ltimes\pi)$, first we extend $f$ to a function from $\mathcal{S}\times\Gamma$ to $\mathcal{S}\times\Lambda$ by setting $f(s,\gamma)=(s,0_\gamma)$ even if $\gamma\not\in\dom(\theta_s)$. This extension of $f$ is continuous, because $f$ is compactly supported on the open subset $\mathcal{S}\ltimes\Gamma$ of the Hausdorff space $\mathcal{S}\times\Gamma$, and the zero map $\gamma\mapsto 0_\gamma$ is continuous from $\Gamma$ to $\Lambda$.

    With this notation we have $f^2(s,\gamma)=0_\gamma$ whenever $\gamma\not\in\dom(\theta_s)$, and $\Phi(f)(s)(\gamma)=f^2(s,\gamma)$ for all $f\in\mathcal{A}(\mathcal{S}\ltimes\pi)$, $s\in\mathcal{S}$ and $\gamma\in\Gamma$.

    We may now proceed to prove that $\Phi(f)$ is continuous. Let $s_0\in\mathcal{S}$ be fixed and consider a pre-basic open subset $[K,V]$ of $\Phi(f)(s_0)$, i.e., $K\subseteq\Gamma$ is compact, $V\subseteq\Lambda$ is open, and $\Phi(f)(s_0)(K)\subseteq V$. 

    This means that for every $k\in K$, we have $f^2(s_0,k)\in V$. As $f^2$ is continuous, there are open neighbourhood $A_k$ and $U_k$ of $s_0$ and $k$, respectively, such that $f^2(A_k\times U_k)\subseteq V$. As $K$ is compact we may consider a finite subcover $\left\{U_1,\ldots,U_n\right\}$ of the $U_k$ (where we write $U_i$ and $A_i$ instead of $U_{k_i}$ and $A_{k_i}$), so let $A=A_1\cap\cdots\cap A_n$. We may now verify that $\Phi(f)(s)(K)\subseteq V$ whenever $s\in A$. Indeed, if $s\in A$ and $k\in K$, then $k\in K_i$ for some $i$, so
    \[\Phi(f)(s)(k)=f^2(s,k)\in f^2(A\times K_i)\subseteq V,\]
    as desired.
\end{enumerate}
In the other direction, we define $\Psi\colon\mathcal{S}\bigstar\mathcal{A}(\pi)\to\mathcal{A}(\mathcal{S}\ltimes\pi)$ as
\[\Psi(f)(s,\gamma)=(s,f(s)(\gamma))\]
Again, we need to verify that $\Psi$ is well-defined, which needs to be done separately in the cases under consideration.

\begin{enumerate}[label=(\Roman*)]
    \item First assume that $\mathcal{S}$ is discrete. Then $\supp(f)$ is finite. Up to taking the finite decomposition $f=\sum_{s\in\supp(f)}f1_{\left\{s\right\}}$ and working on each term separately, we may assume that $\supp(f)=\left\{s\right\}$ for some $s\in\mathcal{S}$. Again up to taking a finite decomposition and working on each term, we may moreover assume that there exists an open Hausdorff $V_s$ such that $f(s)\in C_c(V,\pi)$ for some open Hausdorff subset $V$ of $\Gamma$ with $V\subseteq\dom(\theta^\Gamma_s)$, because $f(s)\in\dom(\Theta_s)$. Then $\Psi(f)\in C_c(\left\{s\right\}\times V,\mathcal{S}\ltimes\pi)$.

    \item Assume now that $\Gamma$ and $\mathcal{S}$ are Hausdorff and $\mathcal{A}(\pi)$ is endowed with the compact-open topology. We shall prove that $\Psi(f)$ is continuous. Suppose $(s_i,a_i)\to (s,a)$. A basic neighbourhood of $\Psi(f)(s,a)$ has the form $A\times V$, where $A$ and $V$ are neighbourhoods of $s$ and $f(s)(a)$, respectively.

    Choose any compact neighbourhood $K$ of $a$ such that $f(s)(K)\subseteq V$. Since $f$ is continuous and $s_i\to s$, then $f(s_i)\to f(s)$ in the compact-open topology, so $f(s_i)(K)\subseteq V$ for all $i$ large enough. Moreover, $s_i\to s$, so $s_i\in A$ for all $i$ large enough as well. Since $a_i\to a$, then $a_i\in K$ for all $i$ large enough.

    We then have, for large $i$,
\[\Psi(f)(s_i,a_i)=(s_i,f(s_i)(a_i))\in A\times f(s_i)(K)\subseteq A\times V.\]
\end{enumerate}

It is straightforward enough to verify that $\Psi$ and $\Phi$ are inverses to one another, and that $\Psi$ is an algebra homomorphism.\qedhere
\end{proof}

\subsection{Smash products}
Smash products are one of the main constructions in the theory of Hopf algebras. They were used by Cohen and Montgomery in \cite{MR0728711} to obtain a dictionary (commonly called the ``duality theorem'') between the theory of rings graded by finite groups and crossed products, which allows one to relate the graded theory of a ring with its non-graded theory -- e.g.\ by comparing graded and non-graded Jacobson radicals. The definition of smash products may be readily carried over to the case of infinite groups, as done in \cite[Definition 2.1]{MR952080}, \cite[p.\ 301]{MR1048417}.

A different generalization of smash products to the case of infinite groups was considered by Quinn in \cite{MR0805958} (and also appears in \cite[\S 7]{MR2046303}). As proven in \cite[Lemma 2.1]{MR0805958}, ``Quinn's smash product''  contains the usual smash product as an essential ideal. We should also remark that there are other generalizations of smash products to different settings, for example in \cite[Definition 5.4]{arxiv1811.01094} in the context of ``$R$-semicategories'' (which are the semigroupoid analogues of a category enriched over the category of modules over a commutative ring $R$.)

In this section we will review the usual definition of a smash product of graded algebras. Moreover, the theory will be slightly extended to the context of \emph{groupoid graded algebras}. Let $R$ be a fixed ring (possibly non-unital and non-commutative).

Suppose that an $R$-algebra $A=\oplus_{g\in G}A_g$ is graded over a discrete groupoid $G$. Denote by $p_g\colon A\to A_g$ the projection of $A$ onto the homogeneous component $A_g$.

\begin{definition}
    The \emph{smash product} $A\#G$ is the set of formal sums $\sum_{g\in G}a_g\delta_g$, where $a_g$ belongs to $\sum\left\{A_h:\so(h)=\ra(g)\right\}$.

    The $R$-bimodule structure of $A\# G$ is the entrywise one, and the product is the bilinear and balanced extension of the rule
    \[(a\delta_g)(b\delta_h)=
        \begin{cases}
        ap_{gh^{-1}}(b)\delta_h,
            &\text{if }\so(g)=\so(h)\\
        0,  
            &\text{otherwise.}
        \end{cases}\]
\end{definition}

If a product $gh$ is defined in $G$, write $A_g\delta_h=\left\{a\delta_h:a\in A_g\right\}$. Note that, as an $R$-bimodule, $A\# G$ decomposes as an inner direct sum $A\#G=\oplus_{(g,h)\in G^{(2)}}A_g\delta_h$.

The definition above is a clear extension of the usual definition of smash products of group-graded rings (\cite[2B]{MR3781435}), regarded as $\mathbb{Z}$-algebras, since groups are simply groupoids with a single vertex.

For completeness, we prove that this indeed gives a $G$-graded algebra structure to $A\#G$.

\begin{proposition}
    With the structure above, $A\#G$ becomes a $G$-graded algebra, with homogeneous components $(A\#G)_g=\sum_h\left\{A_g\delta_h:h\in G,\ra(h)=\so(g)\right\}$.
\end{proposition}
\begin{proof}
    The only non-trivial part of $A\#G$ being an algebra is the associativity of the product. Consider elements of the form $a_{g'}\delta_g$, $b_{h'}\delta_h$ and $c_{k'}\delta_k$, where $g'g$, $h'h$ and $k'k$ are defined in $G$, $a_{g'}\in A_{g'}$, $b_{h'}\in A_{h'}$ and $c_{k'}\in A_{k'}$.
    
    We use the definition of the product to see that
    \begin{itemize}
        \item $(a_{g'}\delta_gb_{h'}\delta_h)(c_{k'}\delta_k)=0$ whenever $gh^{-1}$ is not defined, or if $gh^{-1}$ is defined but $gk^{-1}$ is not;
        \item $(a_{g'}\delta_g)(b_{h'}\delta_hc_{k'}\delta_k)=0$ whenever $hk^{-1}$ is not defined, or if $hk^{-1}$ is defined but $gk^{-1}$ is not.
    \end{itemize}
    Simple computations in $G$ show that the conditions written above are equivalent, so we may assume that all products $gh^{-1}$, $gk^{-1}$ and $hk^{-1}$ are defined. In this case we have
    \[(a_{g'}\delta_gb_{h'}\delta_h)(c_{k'}\delta_k)=(a_{g'}p_{gh^{-1}}(b_{h'})p_{hk^{-1}}(c_{k'})\delta_k\]
    and
    \[(a_{g'}\delta_g)(b_{h'}\delta_hc_{k'}\delta_k)=a_gp_{gk^{-1}}(b_{h'}p_{hk^{-1}}(c_{k'}))\delta_k.\]
    If $hk^{-1}\neq k'$ then both terms above are zero. We thus may assume $hk^{-1}=k'$, then these terms are respectively
    \[(a_{g'}p_{gh^{-1}}(b_{h'})c_{k'}\delta_k\qquad\text{and}\qquad(a_{g'}\delta_g)(b_{h'}\delta_hc_{k'}\delta_k)=a_gp_{gk^{-1}}(b_{h'}c_{k'})\delta_k.\]
    Note that $h'\neq gh^{-1}$ if and only if $h'k'\neq gk^{-1}$, and in this case both terms are zero (because $b_{h'}c_{k'}\in A_{h'k'}$.
    
    In the last case, we have $h'=gh^{-1}$ and $h'k'=gk^{-1}$, in which case both terms are simply $a_{g'}b_{h'}c_{k'}$.
    
    As for $A\#G$ being graded, it should be clear that $A\#G=\oplus_{g\in G}(A\#G)_g$.
    
    Let us prove that $(A\#G)_g(A\#G)_h\subseteq(A\#G)_{gh}$ whenever $(g,h)\in G^{(2)}$. For this, it is enough to consider elements of the form $a_g\delta_m\in(A\#G)_g$ and $a_h\delta_n\in(A\#G)_h$, where $gm$, $hn$ and $mn^{-1}$ are defined in $G$. Then
    \[(a_g\delta_m)(a_h\delta_n)=a_gp_{mn^{-1}}(a_h)\delta_n.\]
    Since $a_h$ belongs to the homogeneous component $A_h$, this product is zero whenever $mn^{-1}\neq h$. On the other hand, if $mn^{-1}=h$ then this product is $a_ga_h\delta_n$, which belongs to $A_gA_h\delta_n\subseteq A_{gh}\delta_n\subseteq(A\#G)_{gh}$.
\end{proof}

We may also be slightly more formal and define $A\#G$ as the set of finitely supported functions $\alpha\colon G\to A$, satisfying $\alpha(g)\in\sum_h\left\{A_h:\so(h)=\ra(g)\right\}$ for all $g\in G$. The $R$-bimodule structure is the pointwise one, and the product is given by
\[\alpha\beta(g)=\sum_{h:\so(h)=\so(g)}\alpha(h)p_{hg^{-1}}(\beta(g)).\]

Now suppose that $\Gamma$ is a semigroupoid, graded by the groupoid $G$ via a homomorphism $d\colon\Gamma\to G$. We may perform a construction analogous to a semidirect product as follows: Let $U(G)$ be the underlying set of the groupoid $G$. Consider the action of left multiplication of $G$ on $U(G)$: For all $g\in G$ and all $h\in U(G)$ with $\so(g)=\ra(h)$, define $g\cdot h=gh\in U(G)$. We may compose this action with the homomorphism $d\colon \Gamma\to G$ and obtain an ``action'' of $\Gamma$ on $U(G)$ (where by an action of a semigroupoid $\Gamma$ we mean simply a homomorphism from $\Gamma$ to the semigroup of partial bijections of $U(G)$). We then define the skew product $\Gamma\#_d G\defeq\Gamma\ltimes U(G)$ just as in Definition \ref{def:semidirect.product}. Namely,
\[\Gamma\#_d G=\left\{(\gamma,g)\in\Gamma\times G:\so(d(\gamma))=\ra(g)\right\},\]
with source and range maps $\so,\ra\colon\Gamma\#_d G\to\Gamma^{(0)}\times G$
\[\so(\gamma,g)=(\so(\gamma),g),\qquad\text{and}\qquad\ra(\gamma,g)=(\ra(\gamma),d(\gamma)g)\]
and product
\[(\gamma_1,g_1)(\gamma_2,g_2)=(\gamma_1\gamma_2,g_2)\qquad\text{ whenever }\so(\gamma_1)=\ra(\gamma_2)\text{ and }g_1=d(\gamma_2)g_2.\]

Then $\Gamma\#_d G$ is a semigroupoid, graded by $G$ via $\widetilde{d}(\gamma,g)=d(\gamma)$.

Now suppose that $\pi\colon\Lambda\to\Gamma$ is an $R$-bundle, and $\Gamma$ is $G$-graded via $d\colon\Gamma\to G$. Then $\Lambda$ is also $G$-graded, via $d\circ\pi$, so we may construct the new bundle
\[\pi\#_dG\colon \Lambda\#_{\pi\circ d} G\to\Gamma\#_d G,\qquad (\pi\# G)(\lambda,g)=(\pi(\lambda),g)\]
which is $G$-graded via $\widetilde{d}$.

We thus obtain the wide generalization of \cite[Theorem 3.4]{MR3781435}.

\begin{theorem}\label{thm:isomorphismofsmashproduct}
    Let $\pi\colon\Lambda\to\Gamma$ be a continuous $R$-bundle, where $R$ is a unital topological groupoid and $\Gamma$ is locally $R$-unital, and suppose that $\Gamma$ is graded over a discrete groupoid $G$ via a homomorphism $d\colon\Gamma\to G$.
    
    Then there exists a $G$-graded isomorphism of algebras
    \[T\colon\mathcal{A}(\pi)\# G\to\mathcal{A}(\pi\#_d G)\]
    given by
    \[T(\alpha\delta_g)(\gamma,h)=\begin{cases}
    (\alpha(\gamma),g),&\text{if }g=h,\\
    (0_\gamma,h),&\text{otherwise},
    \end{cases}\ntag\label{eq:thm:isomorphismofsmashproduct}\]
    for all $g\in G$, all $\alpha\in\sum_k\left\{\mathcal{A}(\pi)_k:\so(k)=\ra(g)\right\}$, and all $(\gamma,h)\in \Gamma\#_d G$.
\end{theorem}
\begin{proof}
    The argument is essentially the same as the one in Theorem \ref{thm:isomorphism.of.naive.crossed.product}. Namely, an element of $\mathcal{A}(\pi\#_d G)$ is simply a function from (a subset of) $\Gamma\times G$ to $\Lambda\times G$ which preserves the second coordinate, so it may be seen simply as a function from (a subset of) $\Gamma\times G$ to $\Lambda$, i.e., an element of the function space $\Lambda^{\Gamma\times G}$.
    
    On the other hand, an element of $\mathcal{A}(\pi)\# G$ is a function $\alpha$ from $G$ to $\mathcal{A}(\pi)$, which is a subset of the function space $\Lambda^\Gamma$, i.e., $\alpha\in(\Lambda^\Gamma)^G$.
    
    The natural function $T$ given in Equation \eqref{eq:thm:isomorphismofsmashproduct} is simply a realization of the natural isomorphism of hom-sets $(\Lambda^\Gamma)^G\to\Lambda^{\Gamma\times G}$ in the category of sets and functions, and is readily verified to be a surjective $G$-graded homomorphism. The inverse of $T$ is given similarly, with appropriate extensions by zero, as in Theorem \ref{thm:isomorphism.of.naive.crossed.product}.\qedhere
\end{proof}

\subsection{Quotients and sectional algebras}\label{subsec:quotients}

The last construction we consider are quotients. Namely, we will prove that ``quotients and sectional algebras commute'', in the sense that if a bundle $\pi/\!\!\sim$ is a quotient of a bundle $\pi$, then the sectional algebra $\mathcal{A}(\pi/\!\!\sim)$ is a quotient of $\mathcal{A}(\pi)$ (in a natural manner). Moreover, up to technical conditions we may determine precisely the ideal $\mathcal{I}$ yielding the natural isomorphism $\mathcal{A}(\pi)/\mathcal{I}\cong\mathcal{A}(\pi/\!\!\sim)$.

We start by recalling the relevant definitions and elementary results.

\begin{denv*}{Final topologies}
    Let $(X,\tau_X)$ be a topological space, $Y$ a set and $f\colon X\to Y$ a function. The \emph{final topology} $\tau_f$ induced by $f$ is the finest topology on $Y$ which makes $f$ is continuous. Explicitly, $\tau_f$ consists of all subsets $U$ of $Y$ such that $f^{-1}(U)\in\tau_X$.
    
    Continuous functions from $(Y,\tau_f)$ may be determined as follows: If $(Z,\tau_Z)$ is any topological space and $g\colon Y\to Z$ is any function, then $g$ is continuous from $(Y,\tau_f)$ to $(Z,\tau_Z)$ if and only if the composite $g\circ f$ is continuous from $(X,\tau_X)$ to $(Z,\tau_Z)$. Thus we may determine continuous functions from $(Y,\tau_f)$ simply in terms of continuous functions from $(X,\tau_X)$. We may restate this in terms of commutative diagrams: given a commutative diagram
    \[\begin{tikzpicture}
    \node (X) at (0,0) {$(X,\tau)$};
    \node (Y) at (0,-1) {$(Y,\tau_f)$};
    \node (Z) at (2,-0.5) {$(Z,\tau_Z)$,};
    \draw[->] (X)--(Y) node[midway,left] {$f$};
    \draw[->] (X)--(Z) node[midway,above] {$h$};
    \draw[->] (Y)--(Z) node[midway,below] {$g$};
    \end{tikzpicture}\]
    the function $h$ is continuous if and only if $g$ is continuous.
    
    The following special case will be of particular interest, as it allows us to verify continuity of functions more easily: Suppose that $(X,\tau_X)$ and $(Y,\tau_Y)$ are two topological spaces and $f\colon X\to Y$ is surjective, continuous and open (with respect to $\tau_X$ and $\tau_Y$). Then $\tau_Y=\tau_f$, i.e., the original topology of $Y$ is actually the final topology induced by $f$.
\end{denv*}

\begin{denv*}{Open equivalence relations}
    Let $R$ be an equivalence relation on a topological space $X$. We denote by $p_R\colon X\to X/R$ the canonical projection map. We will always endow $X/R$ with the final topology induced by $p_R$, and call the topological space $X/R$ thus obtained the \emph{quotient topological space}. 
    
    Note that the $R$-equivalence class of an element $x\in X$ is $p_R^{-1}(p_R(x))$, and more generally the \emph{$R$-saturation} of a subset $A$ of $X$ is $p_R^{-1}(p_R(A))$ (the set of all elements of $X$ which are $R$-equivalent to some element of $A$).
    
    We say that $R$ is \emph{open} if the saturation of every open subset of $X$ is open, or equivalently if $p_X$ is an open map. Open equivalences are useful since they behave well with respect to products: Indeed, if $R$ is open, then the product map $(p_R\times p_R)\colon X\times X\to (X/R)\times (X/R)$ is surjective, continuous and open (where we endow $(X/R)\times (X/R)$ with the product topology), and so the product topology of $(X/R)\times (X/R)$ is the final topology induced by $p_R\times p_R$
\end{denv*}

\begin{denv*}{Locally trivial equivalence relations}
    A class of equivalence relations which will be of great interest to us are the \emph{locally trivial ones}. Let us say that an equivalence relation $R$ on a topological space $X$ is \emph{locally trivial} if $X$ admits a basis of open subsets $U$ for which if $x,y\in U$ and $(x,y)\in R$, then $x=y$. In simpler words, ``$U$ is locally the identity''.
    
    The following are equivalent: (1) $R$ is locally trivial; (2) the diagonal $\Delta_X=\left\{(x,x):x\in X\right\}$ is contained in the interior of $R$, as a subset of $X\times X$; and (3) the quotient map $p_R\colon X\to X/R$ is locally injective. If $R$ is open, (3) may be substituted by (3') the quotient map $p_R$ is a local homeomorphism.
\end{denv*}

\begin{denv*}{Open rigid congruences and quotients of topological semigroupoids}
If $\Lambda$ is a semigroupoid, then a \emph{rigid congruence} on $\Lambda$ is an equivalence relation $\sim$ on $\Lambda$ which furthermore satisfies:
\begin{enumerate}
    \item If $x\sim y$, then $\so(x)=\so(y)$ and $\ra(x)=\ra(y)$;
    \item If $x_1\sim y_1$, $x_2\sim y_2$ and $\so(x_1)=\ra(x_2)$, then $x_1x_2\sim y_1y_2.$
\end{enumerate}
These properties imply that the quotient $\Lambda/\!\!\sim$ has a canonical semigroupoid structure as follows: Let us denote by $\overline{x}$ the $\sim$-class of an element $x\in\Lambda$. The vertex space $(\Lambda/\!\!\sim)^{(0)}$ is simply the initial vertex space $\Lambda^{(0)}$. The source and range maps of $\Lambda/\!\!\sim$ are given by $\so(\overline{x})=\so(x)$ and $\ra(\overline{x})=\ra(x)$, and products are determined by $\overline{x}\cdot \overline{y}=\overline{xy}$.

If $\Lambda$ is a topological semigroupoid and $\sim$ is an open rigid congruence on $\Lambda$, then the quotient semigroupoid $\Lambda/\!\!\sim$ is also a topological semigroupoid, where we endow the vertex space $(\Lambda/\!\!\sim)^{(0)}=\Lambda^{(0)}$ with its original topology. The requirement of $\sim$ being open is essential to ensure that the product map of $\Lambda/\!\!\sim$ is open. See \cite[Proposition 4.6]{arxiv1902.09375} for details.

Furthermore, if $\Lambda$ is étale, then $\Lambda/\!\!\sim$ is étale as well. More precisely, if $U$ is any open bisection of $\Lambda$, then $p_\sim(U)$ is an open bisection of $\Lambda/\!\!\sim$ and $p_\sim$ restricts to a homeomorphism from $U$ onto $p_\sim(U)$. This in turn implies that if $\Lambda$ is locally $R$-normal, where $R$ is a given unital topological ring, then $\Lambda/\!\!\sim$ is also locally $R$-normal.
\end{denv*}

\begin{denv*}{Quotients of $R$-bundles}
Let $\pi\colon\Lambda\to\Gamma$ be a continuous $R$-bundle, where $R$ is a unital topological ring, $\Gamma$ is étale and locally $R$-normal.

\begin{definition}\label{def:bundlecongruence}
A \emph{bundle congruence} $\sim$ on $\pi$ consists of two open rigid congruences $\sim_\Gamma$ and $\sim_\Lambda$ on $\Gamma$ and $\Lambda$, respectively, such that
\begin{enumerate}[label=(\roman*)]
    \item\label{def:bundlecongruencerelmor} $\pi$ is a $(\sim_\Lambda,\sim_\Gamma)$- morphism, in the sense that if $x\sim_\Lambda y$ in $\Lambda$ then $\pi(x)\sim_\Gamma\pi(y)$ in $\Gamma$.
    \item\label{def:bundlecongruencemodulecong} The $R$-bimodule structure on the fibers of $\pi$ is respected by $\sim_\Lambda$, in the sense that if $x_1\sim_\Lambda x_2$, $y_1\sim_\Lambda y_2$, $\pi(x_1)=\pi(y_1)$ and $\pi(x_2)=\pi(y_2)$, then $x_1+y_1\sim_{\Lambda} y_1+y_2$, and similarly for the left and right actions of $R$.
    \item\label{def:bundlecongruenceroll} If $x,y\in \Lambda$ and $\pi(x)\sim_\Gamma\pi(y)$, then there exists $y'\in\Lambda$ such that $y\sim_\Lambda y'$ and $\pi(x)=\pi(y')$.
\end{enumerate}

We shall drop the subscripts and denote either $\sim_\Lambda$ or $\sim_\Gamma$ simply by $\sim$ whenever no confusion arises.
\end{definition}

We may thus construct a new \emph{quotient bundle} $\pi/\!\!\sim$ as follows: Consider the quotient semigroupoids $\Lambda/\!\!\sim$ and $\Gamma/\!\!\sim$. Denote by $p_\Lambda\colon\Lambda\to\Lambda/\!\!\sim$ (and similarly $p_\Gamma$) the quotient map. By Property \ref{def:bundlecongruencerelmor} above, the map $p_\Gamma\circ\pi\colon\Lambda\to\Gamma/\!\!\sim$ factors uniquely through $\Lambda/\!\!\sim$, i.e, there exists a unique semigroupoid homomorphism $\pi/\!\!\sim\colon\Lambda/\!\!\sim\to\Gamma/\!\!\sim$ such that the following diagram commutes:
\[\begin{tikzpicture}
\node (lambda) at (0,1) {$\Lambda$};
\node (gamma) at (2,1) {$\Gamma$};
\node (lambdasim) at (0,0) {$\Lambda/\!\!\sim$};
\node (gammasim) at (2,0) {$\Gamma/\!\!\sim$};
\draw[->] (lambda)--(gamma) node[midway,above] {$\pi$};
\draw[->] (lambda)--(lambdasim) node[midway,left] {$p_\Lambda$};
\draw[->] (gamma)--(gammasim) node[midway,right] {$p_\Gamma$};
\draw[->] (lambdasim)--(gammasim) node[midway,below] {$\pi/\!\!\sim$};
\end{tikzpicture}.\]
As $p_\Gamma\circ\pi$ is continuous, then $\pi/\!\!\sim$ is also continuous.

For each $\alpha\in\Gamma/\!\!\sim$, we have $(\pi/\!\!\sim)^{-1}(\alpha)=p_\Lambda(\pi^{-1}(p_\Gamma^{-1}(\alpha)))$, and its additive structure is determined as follows: Any two elements of $(\pi/\!\!\sim)^{-1}(\alpha)$ may be written as $p_\Lambda(x)$ and $p_\Lambda(y)$, where $p_\Gamma(\pi(x))\sim_\Gamma\pi(y)$. By Property \ref{def:bundlecongruenceroll}, choose $y'\in\Lambda$ such that $\pi(x)=\pi(y')$ and $y\sim_\Lambda x$, and define
\[p_\Lambda(x)+p_\Lambda(y)=p_\Lambda(x+y').\]
Property \ref{def:bundlecongruencemodulecong} guarantees that this addition depends only on the classes $p_\Lambda(x)$ and $p_\Lambda(y)$, and not on any of the representatives $x$, $y$, or $y'$. Left and right multiplication by $R$ on $(\pi/\!\!\sim)^{-1}(x)$ are determined similarly, as $rp_\Lambda(x)=p_\Lambda(rx)$ and $p_\Lambda(x)r=p_\Lambda(xr)$ for all $r\in R$ and $x\in\Lambda$.

One can easily verify that this indeed determines an $R$-bundle structure for $\pi/\!\!\sim$, which actually follows from the following more general fact: For all $x\in\Lambda$, Property \ref{def:bundlecongruenceroll} implies that the restriction of $p_\Lambda$ to $\pi^{-1}(\pi(x))$ is surjective onto $(\pi/\!\!\sim)^{-1}(p_\Gamma(\pi(x)))$ and is an $R$-bimodule homomorphism.
\end{denv*}

We finish by noting that $p_\Gamma\colon\Gamma\to\Gamma/\!\!\sim$ is a surjective local homeomorphism, so $\Gamma/\!\!\sim$ is locally $R$-normal.

We may now prove our main theorem.

\begin{theorem}\label{thm:quotients}
    Let $\pi\colon\Lambda\to\Gamma$ be a continuous $R$-bundle, where $\Gamma$ is an étale locally $R$-normal semigroupoid. Let $\sim$ be a bundle congruence on $\pi$.
    
    Then the map
    \[T\colon \mathcal{A}(\pi)\to\mathcal{A}(\pi/\!\!\sim),\qquad T(\alpha)(x)=\sum_{\gamma\in x}p_{\Lambda}(\alpha(\gamma))\ntag\label{eq:thm:quotients-defoft}\]
    determines an $R$-algebra homomorphism.
    
    Moreover. if $\sim_{\Lambda}$ is locally trivial, then $T$ is surjective.
\end{theorem}

\begin{remark}
    \begin{itemize}
        \item If $d\colon\Gamma\to G$ is a grading of $\Gamma$ by a discrete semigroupoid $G$ which factors through $\Gamma/\!\!\sim$, then the homomorphism above is graded.
        \item If $\pi$ is a local homeomorphism, then $\sim_\Lambda$ is locally trivial. Indeed, in this case, every $x\in\Lambda$ admits a neighbourhood $U$ such that $\pi$ restricts to a homeomorphism from $U$ onto an open bisection $\pi(U)$ of $\Gamma$. If $y\in U$ and $x\sim_\Lambda y$, then $\pi(x)\sim_\Gamma \pi(y)$. Since $\sim_\Gamma$ is rigid and both $\pi(x)$ and $\pi(y)$ belong to the same bisection $\pi(U)$, then $x=y$. Since $\pi$ is injective on $U$ then $x=y$. Therefore $\pi_\Lambda$ is injective on $U$.
    \end{itemize}
\end{remark}

\begin{proof}[Proof of Theorem \ref{thm:quotients}]
    We first need to prove that $T$ is well-defined. First note that as $\sim_\Gamma$ is rigid, then any of its equivalence classes $\alpha$ is contained in $\so^{-1}(\so(\gamma))\cap\ra^{-1}(\so(\gamma))$, where $\gamma$ is any representative of $\alpha$. Thus the sum in Equation \eqref{eq:thm:quotients-defoft} is actually finite, as we already know that for any $\alpha\in\mathcal{A}(\pi)$, and any $\gamma\in\Gamma$, the set $\left\{a\in\so^{-1}(\so(\gamma)):\alpha(a)\neq0\right\}$ is finite.
    
    The function $T(\alpha)$ is a section of $\pi/\!\!\sim$, but we still need to verify that it belongs to $\mathcal{A}(\pi)$. For this, it is enough to assume that $\alpha\in C_c(V,\pi)$, where $V$ is an open bisection of $\Gamma$. In this case, the diagram
    \[\begin{tikzpicture}
    \node (gamma) at (0,0) {$V$};
    \node (lambda) at ([shift={+(2,0)}]gamma) {$\Lambda$};
    \node (gammasim) at ([shift={+(0,-1)}]gamma) {$p_\Lambda(V)$};
    \node (lambdasim) at ([shift={+(2,0)}]gammasim) {$\Lambda/\!\!\sim$};
    \draw[->] (gamma)--(lambda) node[midway,above] {$\alpha$};
    \draw[->] (gamma)--(gammasim) node[midway,left] {$p_\Gamma$};
    \draw[->] (lambda)--(lambdasim) node[midway,right] {$p_\Lambda$};
    \draw[->] (gammasim)--(lambdasim) node[midway,below] {$T(f)$};
    \end{tikzpicture}\]
    commutes, and $p_\Gamma$ restricts to a homeomorphism from $V$ onto $p_\Gamma(V)$, thus $T(\alpha)\in C_c(p_\Gamma(V),\pi/\!\!\sim)$. As $T$ preserves addition, it is a well-defined map from $\mathcal{A}(\pi)$ to $\mathcal{A}(\pi/\!\!\sim)$.
    
    \textbf{Suppose now that $\sim_\Lambda$ is locally trivial}, so that $p_\Lambda$ is a local homeomorphism, and let us prove that $T$ is surjective.
        
    A basic open subset of $\Gamma/\!\!\sim$ has the form $p_\Gamma(V)$ for some open bisection $V$ of $\Gamma$. Thus it is enough, by Lemma \ref{lem:sectionalalgebraisgeneratedbybasis}, to prove that every $F\in C_c(p_\Gamma(V),\pi/\!\!\sim)$ belongs to the image of $T$.
        
    The main idea we employ to find a preimage of $F$ is that if $p_\Lambda$ were invertible, we could simply define $\alpha=p_\Lambda^{-1}\circ F\circ p_\Gamma$ on $V$ and zero everywhere else, which would be a preimage of $F$. However $p_\Lambda$ is only a local homeomorphism, so we need to perform this procedure locally.
        
    Consider the compact $K\defeq \supp(F)\cap p_\Gamma(V)$. As $F$ is continuous on $p_\Gamma(V)$ then $F(K)$ is compact in $\Lambda/\!\!\sim$. As $p_\Gamma$ is a surjective local homeomorphism, there exist open subsets $U_1,\ldots,U_n$ of $\Lambda$ such that $p_\Lambda$ restricts to a homeomorphism from $U_i$ to $p_\lambda(U_i)$ for each $i$, and $F(K)\subseteq\bigcup_{i=1}^n$.
        
    As $V$ is a bisection and $\sim_\Gamma$ is rigid, then $p_\Gamma$ restricts to a homeomorphism from $V$ to $p_\Gamma(V)$, and in particular $p_\Gamma^{-1}(K)\cap V$ is compact.
        
    Let $f_1,\ldots,f_n\colon V\to R$ be a partition of unity of $p_\Gamma^{-1}(K)\cap V$ subordinate to $p_\Gamma^{-1}(F^{-1}(p_\Lambda(U_i)))\cap V$, $i=1,\ldots,n$. As usual we may assume that each $f_i$ has compact support.
        
    Define $\alpha_i\colon p_\Gamma^{-1}(F^{-1}(p_\Gamma(U_i)))\cap V\to\Lambda$ as the pointwise product
    \[\alpha_i=f_i(p_\Gamma|_{U_i}^{-1}\circ F\circ p_\Gamma)\]
    and extend $\alpha_i$ as zero everywhere else of $\Gamma$. Then $\alpha_i$ is a section of $\pi$, and continuously and compactly supported on $p_\Gamma^{-1}(F^{-1}(p_\Gamma(U_i)))\cap V$, i.e., $\alpha_i\in C_c(p_\Gamma^{-1}(F^{-1}(p_\Gamma(U_i)))\cap V,\pi)$.
    
    We may now verify that $F=T(\sum_{i=1}^n\alpha_i)$. Indeed, both $F$ and $T(\sum_{i=1}^n\alpha_i)$ are zero outside of $p_\Gamma(V)$, so we only need to verify that $F(p_\Gamma(v))=T(\sum_{i=1}^n\alpha_i)(p_\Gamma(v))$ for each $v\in V$.
    
    Note that, since each $\alpha_i$ is zero outside of $V$ and $\sim_\Gamma$ is rigid, then in fact each $v\in V$ is the only element of $\Gamma$ which is $\sim_\Gamma$-equivalent to $v$, and on which $\alpha$ is possibly nonzero, so
    \[T(\alpha_i)(p_\Gamma(v))=p_\Lambda(\alpha_i(v))=f_i(v)F(p_\Gamma(v))\]
    Summing over $i$ we obtain
    \[T(\sum_{i=1}^n\alpha_i)(p_\Gamma(v))=\sum_{i=1}^n f_i(v) F(p_\Gamma(v))\]
    Since $p_\Gamma$ is a homeomorphism from $V$ to $p_\Gamma(V)$ then $v\not\in p^{-1}(K)$ if and only if $p_\Gamma(v)\not\in K$, in which case we obtain $F(p_\Gamma(v))=0$, so
    \[T(\sum_{i=1}^n\alpha_i)(p_\Gamma(v))=0=F(p_\Gamma(v)).\]
    Otherwise we have $v\in p_\Gamma^{-1}(K)$, on which $f_1,\ldots,f_n$ are a partition of unity, so
    \[T(\sum_{i=1}^n\alpha_i)(p_\Gamma(v))=\left(\sum_{i=1}^n f_i(v)\right)F(p_\Gamma(v))=F(p_\Gamma(v)).\]
    In any case, we conclude that $F=T(\sum_{i=1}^n\alpha_i)$.\qedhere
\end{proof}

We now want to connect the map $T$ above to (non-naïve) crossed products of inverse semigroups, and for this we will need to determine the kernel of $T$ more precisely. On one hand, the kernel of $T$ is, by definition,
\[\ker T=\left\{\alpha\in\mathcal{A}(\pi):\forall\gamma\in\Gamma, \sum_{\delta\sim \gamma}\alpha(\delta)\sim 0\right\}.\]

We may realize the intuitive idea that if two sections $\alpha,\beta\in\mathcal{A}(\pi)$ ``are the same up to $\sim_\Gamma$ and $\sim_\Lambda$-equivalence classes'', then they define the same element of $\mathcal{A}(\pi/\!\!\sim)$ as follows: suppose that $\alpha,\beta\in\mathcal{A}(\pi)$ satisfy the following property: There exists a bijection $\varphi\colon \supp\beta\to\supp\alpha$ such that $\gamma\sim_\Gamma\varphi(\gamma)$ and $\alpha(\gamma)\sim_\Lambda\beta(\varphi(\gamma))$ for all $\gamma\in\Gamma$. Then $T(\alpha)=T(\beta)$, so $\alpha-\beta\in\ker(T)$.

We may formalize how $\beta$ above may be constructed from $\alpha$ as follows: Suppose that
\begin{itemize}
    \item $\alpha\in C_c(V,\pi)$ for some open Hausdorff set $V$.
    \item $W$ is an open subset of $\Gamma$;
    \item $\varphi\colon W\to V$ is a homeomorphism such that $\gamma\sim_\Gamma\varphi(\gamma)$ for all $\gamma\in W$;
    \item $A$ and $B$ are open subsets of $\Lambda$ such that $\alpha(V)\subseteq A$; and
    \item $\psi\colon A\to B$ is a homeomorphism such that $x\sim_\Lambda\psi(x)$ for all $x\in A$, and such that $\pi\circ\psi\circ\alpha\circ\varphi=\id_W$.
\end{itemize}

Define $\psi\alpha\varphi\colon\Gamma\to\Lambda$ as
\[\psi\alpha\varphi=\psi\circ\alpha\circ\varphi\text{ on }W\quad\text{and}\quad 0\text{ on }\Lambda\setminus W.\]

\begin{definition}\label{def:conjugatesection}
    We say that $\psi\alpha\varphi$ is the section \emph{conjugate to $\alpha$} (via $\varphi$ and $\psi$).
\end{definition}

Note that $\psi\alpha\varphi$ is indeed a section of $\pi$, and it is zero outside $W$ and continuous on $W$. However $\psi\alpha\varphi$ does not necessarily belong to $C_c(W,\pi)$, since it might not be zero outside of a compact contained in $W$.

So in order to guarantee that $\psi\alpha\varphi\in C_c(W,\pi)$, we need to ensure that $\psi$ takes zeroes to zeroes; This is the case, for example, if the ``zero-set'' $\mathbf{0}(\Gamma)=\left\{0_\gamma:\gamma\in\Gamma\right\}$ is $\sim_\Lambda$-saturated; That is, if $x\sim 0_\gamma$ for some $\gamma\in\Gamma$, then $x=0_{\pi(x)}$. This is equivalent to state that $\sim_\Lambda$ restricts to the identity relation on each $R$-bimodule $\pi^{-1}(\gamma)$.

\begin{theorem}\label{thm:kernelquotients}
    Let $\pi\colon\Lambda\to\Gamma$, $\sim$ and $T$ be as in Theorem \ref{thm:quotients}. Suppose moreover that $\pi$ is a local homeomorphism, and that $\mathbf{0}(\Gamma)$ is $\sim_\Lambda$-saturated.
       
    Then the kernel of $T$ is generated as an additive group by the set $\mathbf{K}$ of all sections of the form $\alpha-\psi\alpha\varphi$, where $\psi\alpha\varphi$ is conjugate to $\alpha$ as in Definition \ref{def:conjugatesection}.
\end{theorem}

The determination of $\ker T$ above has some interesting consequences. For example, as $\sim_\Gamma$ is an open equivalence relation on $\Gamma$, then $\sim_\Gamma$ is actually an étale principal groupoid, when endowed with the subspace topology coming from $\Gamma\times\Gamma$. The map $\varphi$ as in the definition of conjugate sections  is nothing more than an element of the ``full semigroup'' (of open bisections) of $\sim_\Lambda$, and similarly for $\sim_\Gamma$ (assuming that $\pi$ is a local homeomorphism). This semigroup is considered, for example, in the non-commutative Stone Duality of Lawson and Lenz (see \cite{MR3077869})..

Thus $\ker T$ may be seen as version of the ``coboundary group'' $B_\varphi=\left\{f-f\circ\varphi^{-1}:f\in C(X,\mathbb{Z})\right\}$ associated to a self-homeomorphism $\varphi$ of a compact Hausdorff space $X$. This group plays a prominent role in the work of Giordano-Putnam-Skau on the structure of Cantor minimal systems (see \cite{MR1363826,MR1710743}. 

This determination of $\ker(T)$ will also allow us to recover the main theorem of \cite{arxiv1804.00396} (see Corollary \ref{cor:mainestresult}).

\begin{proof}[{Proof of Theorem \ref{thm:kernelquotients}}]
    Let us denote by $\langle\mathbf{K}\rangle$ the additive group generated by $\mathbf{K}$. On one hand, it is easy to check that every element of $\mathbf{K}$ belongs to $\ker(T)$, so $\langle\mathbf{K}\rangle\subseteq\ker(T)$, and we just need to prove the reverse inclusion.
    
    An element $\alpha$ of $\ker T$ may be written as a sum $\alpha=\sum_{i=1}^n\alpha_i$, where $\alpha_i\in C_c(V_i,\pi)$ for certain bisections $V_i\in\mathbf{B}(\Gamma)$. We thus proceed by induction on $n$ in order to prove that $\alpha\in\mathbf{K}$.
    
    \begin{itemize}
        \item First suppose that $n=1$. This means that $\alpha(v)\sim 0_v$ for all $v\in V$, and thus $\alpha=0$ because $\mathbf{0}(\Gamma)$ is saturated. Thus $\alpha\in\langle\mathbf{K}\rangle$ trivially.
        \item Suppose the result holds for $n$ and let us prove it for $n+1$.
        
        Suppose that $\alpha=\sum_{i=1}^n\alpha_i+\beta\in\ker(T)$, where $\alpha_i\in C_c(V_i,\pi)$ and $\beta\in C_c(V,\pi)$ for certain bisections $V_i,V$ of $\Gamma$.
        
        Since $\pi$ is a local homeomorphism and $\pi\circ \mathbf{0}=\id_\Gamma$, then the zero section $\mathbf{0}$ is a local homeomorphism as well, and in particular it is an open map. The zero set $\mathbf{0}(\Gamma)$ is thus open in $\Lambda$, and since $\beta$ is continuous on $V$ we obtain
        \[\supp\beta\cap V=\left\{v\in V:\beta(v)\neq 0\right\}.\]
        
        Let $v\in\supp\beta\cap V$. Then $\beta(v)\neq 0$, so $T(\beta)(p_\Gamma(v))\neq 0$ as well because $\mathbf{0}(\Gamma)$ is saturated. However $\sum_{i=1}^n\alpha_i+\beta\in\ker(T)$, so there exists $i\in\left\{1,\ldots,n\right\}$ and $v_i\in \Gamma$ such that $p_\Gamma(v_i)=p_\Gamma(v)$ and $\alpha_i(v_i)\neq 0$. In particular, $v_i\in V_i$.
        
        The map $p_\Gamma\colon\Gamma\to\Gamma/\!\!\sim_\Gamma$ is a local homeomorphism and $p_\Gamma(v)=p_\Gamma(v_i)$, so there exist neighbourhoods $W$ of $v$ and $W_i$ of $v_i$, contained in $V$ and $V_i$, respectively, such that $p_\Gamma$ restricts to homeomorphisms on $W$ and on $U$ with the same image in $\Gamma/\!\!\sim_\Gamma$.
        \[\begin{tikzpicture}
        \node (Wi) at (0,0) {$U$};
        \node (W) at (0,-1) {$W$};
        \node (gammasim) at (4,-0.5) {$p_\Gamma(W)=p_\Gamma(W_i)$,};
        \draw[dashed, ->] (Wi)--(W) node[midway,left] {$\phi$};
        \draw[->] (Wi)--(gammasim) node[midway,above] {$p_\Gamma$};
        \draw[->] (W)--(gammasim) node[midway,below] {$p_\Gamma$};
        \end{tikzpicture}\]
        Both arrows labelled as $p_\Gamma$ above are homeomorphisms of their domains, so $\phi=(p_\Gamma)|_W^{-1}\circ(p_\Gamma)|_U$ defines a homeomorphism $U\to W$ making the diagram above commute.
        
        As $\alpha$ is an open function on $V$, consider the open set $A\defeq\alpha(V)$ of $\Lambda$.
        
        Now consider the points $x=\alpha_i(v_i)$ and $y=\beta(v)$ of $\Lambda$. We have $\pi(x)=v_i$ and $\pi(y)=v$, which are $\sim_\Gamma$-equivalent. Property \ref{def:bundlecongruenceroll} in the definition of a bundle congruence implies that there exists $y'\in\Lambda$ such that $y'\sim_\Lambda y=\beta(v)$, and $\pi(y')=\pi(x)=v_i$. Again using that $\pi$ is a local homeomorphism and making $U$ and $W$ smaller if necessary, we may find a neighbourhood $B$ of $y'$ such that $\pi$ restricts to a homeomorphism from $B$ to $U$, and such that $p_\Lambda$ is a homeomorphism from $B$ onto $p_\Lambda(\alpha(V))$
        
        In short, we have the commutative diagram (with solid lines)
        \[\begin{tikzpicture}
        \node (yprime) at (0,0) {$B$};
        \node (Lambda) at ([shift={+(2,0)}]yprime) {$A$};
        \node (Lambdasim) at ([shift={+(2,0)}]Lambda) {$p_\Lambda(A)$};
        \node (U) at ([shift={+(0,-1)}]yprime) {$U$};
        \node (W) at ([shift={+(2,0)}]U) {$W$};
        \node (Gammasim) at ([shift={+(2,0)}]W) {$p_\Gamma(W)$};
        \draw[->] (yprime)--(U) node[midway,left] {$\pi$};
        \draw[->] (U)--(W) node[midway,above] {$\varphi$};
        \draw[->] (W)--(Lambda) node[midway,left] {$\beta$};
        \draw[->] (Lambda)--(Lambdasim) node[midway,above] {$p_\Lambda$};
        \draw[->] (Lambdasim)--(Gammasim) node[midway,right] {$\pi/\!\!\sim$};
        \draw[->] (W)--(Gammasim) node[midway,above] {$p_\Gamma$};
        \draw[->] (U) to[out=330,in=210] node[midway,below] {$p_\Gamma$} (Gammasim);
        \draw[->] (yprime) to[out=30,in=150] node[midway,above] {$p_\Lambda$} (Lambdasim);
        \draw[->,dashed] (Lambda)--(yprime);

        \node (yprimep) at ([shift={+(2,0)}]Lambdasim) {$y'$};
        \node (Lambdap) at ([shift={+(2,0)}]yprimep) {$\alpha(v)$};
        \node (Lambdasimp) at ([shift={+(2,0)}]Lambdap) {$p_\Lambda(\alpha(v))$};
        \node (Up) at ([shift={+(0,-1)}]yprimep) {$v_i$};
        \node (Wp) at ([shift={+(2,0)}]Up) {$v$};
        \node (Gammasimp) at ([shift={+(2,0)}]Wp) {$p_\Gamma(v)$};
        \draw[|->] (yprimep)--(Up) node[midway,left] {$\pi$};
        \draw[|->] (Up)--(Wp) node[midway,above] {$\varphi$};
        \draw[|->] (Wp)--(Lambdap) node[midway,left] {$\beta$};
        \draw[|->] (Lambdap)--(Lambdasimp) node[midway,above] {$p_\Lambda$};
        \draw[|->] (Lambdasimp)--(Gammasimp) node[midway,right] {$\pi/\!\!\sim$};
        \draw[|->] (Wp)--(Gammasimp) node[midway,above] {$p_\Gamma$};
        \draw[|->] (Up) to[out=330,in=210] node[midway,below] {$p_\Gamma$} (Gammasimp);
        \draw[|->] (yprimep) to[out=30,in=150] node[midway,above] {$p_\Lambda$} (Lambdasimp);
        \end{tikzpicture}
        \]
        where all sets in the left diagram are open in $\Gamma$, $\Lambda$, $\Gamma/\!\!\sim$ or $\Lambda/\!\!\sim$ accordingly, and all arrows are homeomorphisms between their sources and ranges. If we define $\psi\colon A\to B$ to be the unique function which, when placed in the dashed space above, makes the resulting diagram commutative, then the commutativity of this diagram means precisely that $\varphi$ and $\psi$ satisfy the necessary conditions as in the definition of conjugate sections.
        
        The idea is now to perform this procedure locally, with the usage of partitions of unity, which will slightly modify the diagram above.
        
        We thus perform the procedure above for all $v\in\supp(\beta)\cap V$ and use its compactness in order to find a finite collection of tuples of sets $(U_j,W_j,A_j,B_j)$, $j=1,\ldots,N$, which makes the analogous diagram as the one above commutative, and such that
        \begin{itemize}
            \item $\supp\beta\cap V\subseteq\bigcup_{j=1}^N W_j$;
            \item For each $j$, $W_j\subseteq V$ and $U_j\subseteq V_i$ for some $i$.
        \end{itemize}
        
        Let $\left\{f_j:j=1,\ldots,N\right\}$ be a partition of unity of $\supp(\beta)\cap V$ subordinate to $W_1,\ldots,W_N$. We may break down $\beta$ as
        $\beta=\sum_{j=1}^N f_j\beta$, where each section $f_j\beta$ is supported on $W_j$, i.e., $f_j\beta\in C_c(W_j,\pi)$.
        
        Given $j$, Let $f_jA_j=(f_j\beta_j)(W_j)$ and $f_jB_j=\left\{f_j(\varphi(\pi(b)))b:b\in B_j\right\}$. Then $f_jA_j$ and $f_jB_j$ are also open in $\Lambda$, and the following diagram (with solid arrows) commutes:
        \[\begin{tikzpicture}
        \node (yprime) at (0,0) {$f_jB_j$};
        \node (Lambda) at ([shift={+(2,0)}]yprime) {$f_jA_j$};
        \node (Lambdasim) at ([shift={+(2,0)}]Lambda) {$p_\Lambda(f_jA_j)$};
        \node (U) at ([shift={+(0,-1)}]yprime) {$U_j$};
        \node (W) at ([shift={+(2,0)}]U) {$W_j$};
        \node (Gammasim) at ([shift={+(2,0)}]W) {$p_\Gamma(W_j)$};
        \draw[->] (yprime)--(U) node[midway,left] {$\pi$};
        \draw[->] (U)--(W) node[midway,above] {$\varphi$};
        \draw[->] (W)--(Lambda) node[midway,left] {$\beta$};
        \draw[->] (Lambda)--(Lambdasim) node[midway,above] {$p_\Lambda$};
        \draw[->] (Lambdasim)--(Gammasim) node[midway,right] {$\pi/\!\!\sim$};
        \draw[->] (W)--(Gammasim) node[midway,above] {$p_\Gamma$};
        \draw[->] (U) to[out=330,in=210] node[midway,below] {$p_\Gamma$} (Gammasim);
        \draw[->] (yprime) to[out=30,in=150] node[midway,above] {$p_\Lambda$} (Lambdasim);
        \draw[->,dashed] (Lambda)--(yprime);
        \end{tikzpicture}
        \]
        where all sets are open in $\Gamma$, $\Lambda$, $\Gamma/\!\!\sim$ or $\Lambda/\!\!\sim$, and all arrows are homeomorphisms. Let $\psi_j\colon f_j A_j\to f_jB_j$ be the homeomorphism associated to the dashed arrow above making the diagram commute. Then $\varphi_j$ and $\psi_j$ satisfy the required conditions to define the conjugate section $\psi_j(f_j\beta)\phi_j$. This conjugate section, in turn, will belong to some $C_c(V_i,\pi)$, (namely, just cchoose $i$ such that $U_j\subseteq V_i$).
        
        We may thus rewrite
        \[\sum_{i=1}^n\alpha_i+\beta=\sum_{i=1}^n\alpha_i+\sum_{j=1}^N\psi_j(f_j\beta)\varphi_j+\sum_{j=1}^n\left[(f_j\beta)-(\psi_j(f_j\beta)\varphi_j)\right].\]
        The last term of the right-hand side belongs to $\langle\mathbf{K}\rangle$, and in particular to $\ker T$. The remainder terms of the right-hand sides are may be rewritten as a sum, with at most $n$ elements, of elements of $C_c(V_i,\pi)$, $i=1,\ldots,n$. Namely, for each $j$ choose $I(j)$ such that $U_j\subseteq V_{I(j)}$. Then
        \[\sum_{i=1}^n\alpha_i +\sum_{j=1}^N\psi_j(f_j\beta)\varphi_j=\sum_{i=1}^n\left(\alpha_i+\sum_{j:I(j)=i}\psi_j(f_j\beta)\varphi_j\right)\]
        Since this also belong to $\ker(T)$, the induction hypothesis implies that it belogns to $\langle\mathbf{K}\rangle$, so we conclude that $\sum_{i=1}^n\alpha_i+\beta\in\langle\mathbf{K}\rangle$, as desired.
    \end{itemize}
\end{proof}

We may now apply the previous theorems to obtain a far-reaching generalization of \cite[Theorem 5.10]{arxiv1804.00396}. First a matter of notation: if $S$ is a discrete inverse semigroupoid with a $\land$-preaction on an algebra $A$, and $s\in S$ and $a\in\dom(\theta_s)$ we define the element $\delta_sa$ of the naïve crossed product $S\bigstar A$ as the function which takes $s$ to $a$ and all other elements of $S$ to $0$.

We may adapt the notion of crossed product of \cite{MR2559043}, and call the ``classical'' crossed product of $S$ and $A$ the quotient of $S\bigstar A$ by the ideal generated by terms of the form $\delta_sa-\delta_ta$, where $a\in\dom(\theta_s)\cap\dom(\theta_t)$ and $s\leq t$.

\begin{corollary}\label{cor:mainestresult}
    Let $S$ be a discrete inverse semigroupoid, $\mathcal{G}$ be an ample groupoid, $R$ a discrete unital ring, $A$ a discrete $R$-algebra, and $\theta\colon\mathcal{S}\curvearrowright\mathcal{G}$ a continuous, associative and open $\land$-preaction of $\mathcal{S}$ on $\mathcal{G}$.
    
    Consider the semidirect product semigroupoid $S\ltimes\mathcal{G}$, and let $\mathscr{G}(S\ltimes\mathcal{G})$ be its \emph{initial groupoid}, i.e., the quotient of $S\ltimes\mathcal{G}$ by the relation
    \[(s_1,g_1)\sim(s_2,g_2)\iff g_1=g_2\text{ and there exists }u\leq s_1,s_2\text{ such that }g_1\in\dom(\theta_u).\]
    (This is also called the \emph{groupoid of germs} of the $\land$-preaction $\theta$.)
    
    Let $\Theta$ be the $\land$-preaction of $S$ on $A\mathcal{G}$ induced by $\theta$: given $s\in S$, set
    \[\dom(\Theta_s)=\left\{f\in A\mathcal{G}:f=0\text{ outside } \dom(\theta_s)\right\},\]
    and define $\Theta_s(f)=f\circ\theta_s^{-1}$ on $\ran(\theta_s)$ and $0$ everywhere else.
    
    Then $A(\mathscr{G}(S\ltimes\mathcal{G}))$ is isomorphic to the quotient of $S\bigstar(A\mathcal{G})$ by the ideal generated by sections $\delta_sa-\delta_ta$, where $a\in\dom(\Theta_s)\cap\dom(\Theta_t)$ and $s\leq t$.
\end{corollary}
\begin{proof}
    Consider the bundles $\pi\colon A\times\mathcal{G}\to\mathcal{G}$ and $\xi\colon A\times(S\ltimes\mathcal{G})\to S\ltimes\mathcal{G}$ given by the obvious coordinate projections.
    
    The $\land$-preaction of $S$ on $\mathcal{G}$ ``extends'' naturally to a $\land$-preaction of $S$ on $A\times\mathcal{G}$, also denoted by $\theta$, given by $\theta_s(a,g)=(a,\theta_s(g))$ whenever the right-hand side makes sense.
    
    In this manner, we may identify $A\times(S\ltimes\mathcal{G})$ with $S\ltimes(A\times\mathcal{G})$, and $\xi$ with $S\ltimes\pi$. We thus obtain, from Theorem \ref{thm:isomorphism.of.naive.crossed.product}, an explicit isomorphism
    \[S\bigstar A\mathcal{G}\cong \mathcal{A}(\xi)\]
    
    Now we need to take care of quotients. We already have a relation $\sim$ on $\mathcal{S}\ltimes\mathcal{G}$, so consider the associated relation (again denoted $\sim$) on $\mathcal{S}\ltimes(A\times\mathcal{G})$ as
    \[(s_1,a_1,g_1)\sim(s_2,a_2,g_2)\iff a_1=a_2\text{ and }(s_1,g_1)\sim(s_2,g_2)\]
    The only non-trivial property of Definition \ref{def:bundlecongruence} is \ref{def:bundlecongruenceroll}, which is proven as follows: If $x=(s_1,a_1,g_1)$ and $y=(s_2,a_2,g_2)$ have equivalent images under $\xi$, then $(s_1,g_1)\sim(s_2,g_2)$, so $y'=(s_1,a_2,g_1)$ has the same image as $x$ under $\xi$, and it is equivalent to $y$.
    
    Then $S\ltimes(A\times\mathcal{G})/\!\!\sim$ is isomorphic, in the obvious manner, to $A\times(S\ltimes G)/\!\!\sim=A\times\mathscr{G}(S\ltimes\mathcal{G})$, and the quotient bundle $\xi/\!\!\sim\colon A\times\mathscr{G}(S\ltimes\mathcal{G})\to\mathscr{G}(S\ltimes\mathcal{G})$ is the coordinate projection. By Theorem \ref{thm:quotients}, we obtain an explicit surjective homomorphism
    \[T\colon S\bigstar A\mathcal{G}\cong\mathcal{A}(\xi)\to \mathcal{A}(\xi/\!\!\sim)=A(\mathscr{G}(S\ltimes\mathcal{G}).\]
    
    We use the explicit description of $\ker T$ given by Theorem \ref{thm:kernelquotients}: Two basic open sets $V'$ and $W'$ of $S\ltimes\mathcal{G}$ will have the forms
    \[V'=\left\{s\right\}\times V\quad\text{and}\quad W'=\left\{t\right\}\times W\]
    for certain $s,t$ and compact-open Hausdorff $V\subseteq\dom(\theta_s)$, $W\subseteq\dom(\theta_t)$.
    
    A homeomorphism $\varphi'$ from $W'$ to $V'$ is of the form $(t,g)\mapsto (s,\varphi(g))$ for some homeomorphism $\varphi\colon W\to V$. However, if $\varphi'$ preserves $\sim$-classes then $\varphi$ is actually the identify, which means that $W=V$, and for every $g\in V$ there is $u_g\leq s,t$ with $g$ in its domain. As $\mathcal{G}$ is ample, we may divide $V$ into finitely many disjoint compact-open $V_1,\ldots,V_n$, an find $u_1,\ldots,u_n$ smaller than $s$ and $t$, with $V_i\subseteq\dom(\theta(u_i))$. Moreover, $A$ is discrete, so we may divide each $V_i$ further and assume that the section $\alpha$, seen as a function from $S\ltimes\mathcal{G}$ to $A$, is constant $a_i$ on $V_i$ and zero everywhere else.
    
    Similarly, the function $\psi$, defined on appropriate domains, is of the form $\psi(r,s,g)=(r,t,g)$. So $\psi\alpha\varphi(t,g)=(a_i,t,g)$ on $\left\{t\right\}\times V_i$.
    
    We now see $\alpha$ and $\psi\alpha\varphi$ as elements of $S\bigstar A\mathcal{G}$, i.e., as functions from $S$ to $A\mathcal{G}$, as in Theorem \ref{thm:isomorphism.of.naive.crossed.product}. Under this realization, it follows that $U_i\subseteq\dom(\theta_{u_i})$, so $a_i1_{U_i}\in\dom(\Theta_{u_i})$ (the function $a_i1_{U_i}$, from $\mathcal{G}$ to $A$, takes $U_i$ to $a_i$ and $\mathcal{G}\setminus U_i$ to $0$). Therefore we obtain
    \[\alpha-\psi\alpha\varphi=\sum_{i=1}^n \delta_s(a_i1_{U_i})-\delta_t(a_i1_{U_i})=\left(\sum_{i=1}^n\delta_s(a_i1_{U_i})-\delta_{u_i}(a_i1_{U_i})\right)+\left(\sum_{i=1}^n\delta_{u_i}(a_i1_{U_i})-\delta_t(a_i1_{U_i})\right).\]
    Since $u_i\leq s,t$ then each term in each sum above belongs to the desired generating set.\qedhere
\end{proof}

\bibliographystyle{amsplain}
\bibliography{library}

\end{document}